\pdfoutput=1
\documentclass[10pt]{article}

\usepackage{rotating}
\usepackage{times}
\usepackage{alltt}
\usepackage{lscape}
\usepackage[english]{babel}
\usepackage{amsmath,amsthm}
\usepackage{amsfonts}
\usepackage{listings}
\usepackage{color}
\usepackage{algorithm}
\usepackage{algpseudocode}
\usepackage{authblk}
\usepackage{hyperref}
\algrenewcommand\algorithmicindent{1.0em}
\usepackage{graphicx}
\usepackage{booktabs}
\usepackage{array}
\usepackage{dblfloatfix}
\newcolumntype{L}[1]{>{\raggedright\let\newline\\\arraybackslash\hspace{0pt}}m{#1}}
\newcolumntype{C}[1]{>{\centering\let\newline\\\arraybackslash\hspace{0pt}}m{#1}}
\newcolumntype{R}[1]{>{\raggedleft\let\newline\\\arraybackslash\hspace{0pt}}m{#1}}

\algrenewcommand\algorithmicindent{1.0em}

\usepackage[multiple]{footmisc}

\usepackage{enumitem}
\usepackage[displaymath, mathlines, right]{lineno}

\usepackage{mathtools}
\DeclarePairedDelimiter\floor{\lfloor}{\rfloor}

\newtheorem{thm}{Theorem}[section]
\newtheorem{cor}[thm]{Corollary}

\theoremstyle{definition}
\newtheorem{defn}[thm]{Definition}
\theoremstyle{remark}

\numberwithin{equation}{section}
\newcommand{\comment}[1]{}
\newcommand{\assign}{:=}
\newcommand{\dsf}{\text{DSF}}
\newcommand{\codsf}{\text{coDSF}}
\newcommand{\rni}[1]{\text{RNI}({#1})}
\newcommand{\ints}{\mathbb{Z}}
\newcommand{\nats}{\mathbb{N}}
\newcommand{\pnats}{\mathbb{N}^{>0}}
\newcommand{\reals}{\mathbb{R}}
\newcommand{\nnreals}{\mathbb{R}^{\ge0}}
\newcommand{\preals}{\mathbb{R}^{>0}}

\newcommand{\pee}{\mathbb{P}}

\newcommand{\abs}[1]{\lvert #1 \rvert}

\newcommand{\lead}[1]{\par\medskip\noindent\textbf{#1.}}

\begin{document}
	
\title{Digit Serial Methods with Applications to Division and Square Root\\ {\large \it (with mechanically checked correctness proofs)}}

\author[1]{Warren~E.~Ferguson~Jr}
\author[2]{Jesse~Bingham}
\author[2]{Levent~Erk\"{o}k}
\author[2]{John~R.~Harrison}
\author[2]{Joe~Leslie-Hurd}
\affil[1]{Intel Corporation, retired.}
\affil[2]{Intel Corporation, Hillsboro, Oregon.}
\renewcommand\Authands{ and }

\maketitle

%
%
%


\begin{abstract}
We present a generic digit serial method (DSM) to compute the digits of a real number $V$.
Bounds on these digits, and on the errors in the associated estimates of $V$ formed from these digits, are derived.
To illustrate our results, we derive such bounds for a parameterized family of high-radix algorithms for division and square root.
These bounds enable a DSM designer to determine, for example, whether a given choice of parameters allows rapid formation and rounding of its approximation to $V$. All our claims are mechanically verified using the HOL-Light theorem prover, and are included in the appendix with commentary.
\\

\end{abstract}

\section{Introduction\label{sec:introduction}}

Let $V$ be a real number. 
A digit serial method (DSM) is an algorithm that determines the digits of $V$ serially, starting with the leading digit. 
A DSM begins by initializing an accumulator to zero and, as each digit is determined, that digit is aligned and added to the accumulator.
Successive values of this accumulator form a sequence of estimates of $V$.

The primary contribution of this paper is a generic DSM analysis method for determining bounds on the magnitudes of the digits, as well as bounds on the error associated with the estimates of $V$.
These bounds allow a designer to determine the required bit-width of signals representing these digits and errors, and to determine when the estimates of $V$ can be efficiently formed and rounded by, say, on-the-fly techniques~\cite{ercegovac1987fly, ercegovac1989fly}.

The major results presented here are the Proxy Theorem~\ref{thm:proxy_theorem} and its Corollary~\ref{cor:proxy_max} with illustrations of their application to division and square root algorithms. 
These results have been checked/formalized using the {HOL Light}~\cite{harrison1996hol} theorem prover; a short extract from the formalization is presented in the appendix.

The analysis of low-radix DSM for division and square root is well-understood~\cite{ercegovac1994division}. 
Analyses of specific high-radix DSM for these operations are described in~\cite{briggs199317, ercegovac1994very, lang1995very}. 
An additional contribution of this paper is the application of our generic DSM analysis to a parameterized family of high-radix DSM algorithms for division and square root.

\section{Scaling\label{sec:scaling}}

The DSM considered here assume that $V \in (0,1)$, so the leading digit of $V$ is known to be the first fraction digit. 
For this assumption to be true, it may be necessary to scale the problem. 
Scaling is a three step process: (1) reduce the general problem to simpler problem by scaling, (2) determine the result of the simpler problem, and (3) reconstruct the general result from the result of the simpler problem. 

For completeness, we briefly describe well-known scalings for division and square root of positive normalized finite precision binary floating-point numbers. Here, a positive normalized finite precision binary floating-point number is a real value of the form $s 2^e$ composed of a normalized significand $s = 1 + f/2^k$, an integer exponent $e$, and a fraction $f/2^k$ where $f$ is a non-negative integer less than $2^k$ for some positive integer $k$.

\lead{Scaling for division} Consider the computation of the quotient $Q \equiv (s_x 2^{e_x})/(s_y 2^{e_y})$ where $s_x$ and $s_y$ are normalized finite precision binary significands, and $e_x$ and $e_y$ are integers.
Scaling reduces the computation of $Q$ to the computation of a related quotient $V \in (0,1)$, a DSM is used to compute $V$, and $Q$ is reconstructed from the value of $V$.
One possible scaling uses the reduction
\[
V \equiv X/Y  \quad\text{where}\quad  (X,Y) \equiv (s_x/2, s_y) ,
\]
so $X \in [1/2,1)$, $Y \in [1,2)$, and $V \in (1/4,1)$. 
After the DSM determines $V$, the final result is reconstructed as follows:
\[
Q = V  2^{e_x - e_y + 1} .
\]

\lead{Scaling for square-root} Consider the computation of the square root $R \equiv \sqrt{s_x 2^{e_x}}$ where $s_x$ is a normalized finite precision binary significand and $e_x$ is an integer.
Scaling reduces the computation of $R$ to the computation of a related square root $V \in (0,1)$, a DSM is used to compute $V$, and $R$ is reconstructed from the value of $V$.
One possible scaling uses the reduction 
\[
V \equiv \sqrt{X} \quad\text{where}\quad
X \equiv \begin{cases*}
s_x/4 & \text{even $e_x$} \\
s_x/2 & \text{odd $e_x$}
\end{cases*},
\]
so $X \in [1/4,1)$ and $V \in [1/2,1)$.
After the DSM determines $V$, the final result is reconstructed as follows:
\[
R = V  \begin{cases*}
			2^{(e_x+2)/2} & \text{even $e_x$} \\
			2^{(e_x+1)/2} & \text{odd $e_x$} 
		 \end{cases*} .
\]

For both division and square root, scaling has reduced the original problem to the computation of a value $V \in (0,1)$, combined with
integer additions that determine the associated exponent.

\section{Basic DSM\label{sec:digitserialmethods}}

Consider the following mixed-radix representation of a real number $V$:
\begin{align*}
V &= \frac{1}{\beta_1}\Big( v_1 + \frac{1}{\beta_2}\Big(v_2 + \frac{1}{\beta_3}\Big(v_3 + \cdots\Big)\Big)\Big) \\
&= \frac{v_{1}}{B_{1}} + \frac{v_{2}}{B_{2}} + \frac{v_{3}}{B_{3}} + \cdots 
\end{align*}
where\footnote{Notation: Reals $\reals$, non-negative reals $\nnreals$, positive reals $\preals$, integers $\ints$, natural numbers $\nats = \{0,1,\ldots\}$ , counting numbers $\pnats = \{1,2,\ldots\}$.} $\forall i \in \pnats: B_i \equiv \beta_1 \beta_2 \ldots \beta_i$.
We always assume that $\{v_i\}_{i=1}^{\infty}$ is a sequence of integers (called \emph{digits}), and that $\{\beta_i\}_{i=1}^{\infty}$ is a sequence of integers (called \emph{radices} or \emph{bases}), each $2$ or greater.
If $B_0 \equiv 1$, then $\forall i \in \nats: B_{i+1} = \beta_{i+1} B_i$.

A DSM accumulates the terms of the series for $V$ serially.
Start with an accumulator initialized to $0$. The terms involving the digits $v_1, v_2, v_3, \ldots$ are then consecutively added to the accumulator.
The values of the accumulator after each digit is added defines the \emph{head} sequence $\{H_i\}_{i=0}^\infty$ where:
\[
H_0 \equiv 0,\; \forall i \in \pnats: H_i \equiv \frac{v_1}{B_1} + \frac{v_2}{B_2} + \cdots + \frac{v_i}{B_i} .
\]
Associated with each head $H_i$ is the \emph{tail} $T_i$ defined as:
\[
\forall i \in \nats: T_i \equiv B_i \left(V - H_i\right) = B_i \left(\frac{v_{i+1}}{B_{i+1}} + \frac{v_{i+2}}{B_{i+2}} + \cdots\right) .
\]
Intuitively, $H_i$ is the approximation to the target result $V$ that has been computed after step $i$,
while $T_i$ is the error in this approximation normalized by $B_i$; here $T_i/B_i$ is analogous to a floating-point value $s 2^e$ with $T_i \sim s$ and $1/B_i \sim 2^e$. 
This definition of the tails provides the invariant $\forall i \in \nats : V = H_i + T_i/B_i$.

We can summarize the above as follows:
\begin{align*}
 &B_0 = 1,\; \forall i \in \nats: B_{i+1} = \beta_{i+1} B_i, \\
 &H_0 = 0,\; \forall i \in \nats: H_{i+1} = H_i + v_{i+1}/B_{i+1}, \;\text{and} \\
 &T_0 = V,\; \forall i \in \nats: T_{i+1} = \beta_{i+1} T_i - v_{i+1} .
\end{align*}

\lead{Digit selection}
In the recurrence
\[
\forall i \in \nats: \beta_{i+1} T_{i} = v_{i+1} + T_{i+1}
\]
note that 
\[
T_{i+1} = \frac{v_{i+2}}{\beta_{i+2}} +
\frac{v_{i+3}}{\beta_{i+2}\beta_{i+3}} +  
\cdots
\]
As we shall see in Sect.~\ref{sec:onthefly}, if the digits satisfy $\forall k \ge 2: \abs{v_k} < \beta_k$,
a simple algorithm can be used to accumulate the digits.
If this condition holds then $\abs{T_{i+1}}$, the distance between
$\beta_{i+1} T_i$ and $v_{i+1}$, is at most $1$.
Consequently, a plausible choice for $v_{i+1}$ is an integer near $\beta_{i+1} T_i$. 

We therefore introduce \emph{digit selection functions} $\forall i \in \pnats: \dsf_i: \reals \rightarrow \ints$ that ``round'' their real argument to a nearby integer, so $\forall i \in \nats: v_{i+1} \equiv \dsf_{i+1}(\beta_{i+1}T_i)$.
Paired with any digit selection function $\dsf$ is the \emph{complementary digit selection function} $\codsf : \reals \rightarrow \reals$ defined as
\[
\forall z \in \reals: \codsf(z) \equiv z - \dsf(z) .
\]
From the partition
\[
\beta_{i+1}T_i = \dsf_{i+1}(\beta_{i+1}T_i) + \codsf_{i+1}(\beta_{i+1}T_i) 
\]
of $\beta_{i+1}T_i$, we recognize that $T_{i+1} =\codsf_{i+1}(\beta_{i+1}T_i)$.

Note that $\abs{\codsf(z)}$ is the distance between $z$ and $\dsf(z)$, or equivalently the error in approximating $z$ by $\dsf(z)$.
It makes sense, then, to classify digit selection functions by the maximum value of $\abs{\codsf(z)}$ for all $z$.

\begin{defn} (Round to Nearby Integer)
	For $\Omega \in \reals$, $\rni{\Omega}$ is the collection of all digit selection functions $\dsf: \reals \rightarrow \ints$ such that $\forall z \in \reals: \abs{\codsf(z)} \le \Omega$.
\end{defn}

We argue that $\rni{\Omega} = \emptyset$ when $\Omega < 1/2$. 
For suppose $\rni{\Omega}$ is nonempty and choose $\dsf \in \rni{\Omega}$. 
When $z = n + 1/2$ for some integer $n$, $\dsf(z)$ is an integer in the interval $[z-\Omega,z+\Omega]$. 
But that is impossible because there are no integers in this interval. Therefore, $\Omega < 1/2$ implies
$\rni{\Omega} = \emptyset$. For this reason we always assume that $\Omega \ge 1/2$. 

When $\dsf \in \rni{\Omega}$ with $1/2 \le \Omega < 1$, $\dsf(z)$ belongs to the interval $[z-\Omega,z+\Omega]$, whose length $2\Omega$ is in the interval $[1, 2)$.
There is always one, and sometimes two, integers in this interval, and $\dsf(z)$ must be one of these integers.

\begin{thm} 
	\label{thm:arni_digit_bound}
	If $v \equiv \dsf(z)$ where $\dsf \in \rni{\Omega}$, then $\abs{v} \le \floor{\abs{z} + \Omega}$.
\end{thm}
\begin{proof}
	Since $\forall x: \abs{\codsf(x)} \le \Omega$ and $v = \dsf(z)$, applying the triangle inequality yields 
	\[
	\abs{v} = \abs{\dsf(z)} = \abs{z - \codsf(z)} \le \abs{z} + \Omega.
	\]
	The result follows by applying the floor function to the inequality and using the fact that $v$ is an integer. 
\end{proof}

\begin{algorithm}[t]
	\caption{Basic DSM that computes $\{(B_i, H_i, T_i)\}_{i=0}^\infty$ for $V \in \reals$ where $\forall i \in \pnats: (\dsf_i \in \rni{\Omega_i}) \wedge (\beta_i\ge2)$.\label{alg:dsm_basic}}
	\begin{algorithmic}[0] 
		\Procedure{DSM\_Basic}{$V$}
		\State $(B_0, H_0, T_0) \assign (1, 0, V)$
		\For{$i \assign 0,1,2,\ldots$}
        \State \{\textbf{Invariant}: $V = H_{i} + T_{i}/B_{i}$\}
		\State $v_{i+1} \assign \dsf_{i+1}(\beta_{i+1} T_{i})$
		\State $B_{i+1} \assign \beta_{i+1} B_{i}$
		\State $H_{i+1} \assign H_{i} + v_{i+1}/B_{i+1}$
		\State $T_{i+1} \assign \beta_{i+1}T_i - v_{i+1}$
		\EndFor
		\EndProcedure
	\end{algorithmic}
\end{algorithm}

Algorithm \ref{alg:dsm_basic} is the result of combining the information presented above.\footnote{See the description of radix-conversion in~\cite{KnuthArt2}.}
For this algorithm, bounds on the absolute error $\abs{T_i}/B_i$ in the estimate $H_i$ of $V$, and on the digit $v_i$, are easy to derive. We know that
\[ 
\abs{T_0} = V \;\;\mbox{and} \;\; \forall i \in \pnats: \abs{T_i} \le \Omega_i
\]
because $\forall i \in \pnats: T_i = \codsf_i(\beta_i T_{i-1})$ where $\dsf_i \in \rni{\Omega_i}$, and so applying Theorem \ref{thm:arni_digit_bound} yields the digit bounds
\begin{align*}
\forall i \in \pnats: \abs{v_i} \le \begin{cases*}
\floor{\beta_1 V + \Omega_1}            & \text{if}\; $i = 1$ \\
\floor{\beta_i \Omega_{i-1} + \Omega_i} & \text{if}\; $i > 1$
\end{cases*} .
\end{align*}
When the sequence $\{\Omega_i\}_{i=1}^\infty$ is bounded, so too is the tail sequence $\{T_i\}_{i=0}^\infty$. 
The following result proves that the head sequence converges to $V$ if the tail sequence is bounded. 

\begin{thm}\label{thm:dsm_bnd_convergence}
Let $V$ and $\{\beta_i\}_{i=1}^\infty$ be given as described in Algorithm \ref{alg:dsm_basic}. If the sequence $\{T_i\}_{i=0}^\infty$ is bounded, then the sequence $\{H_i\}_{i=0}^\infty$ converges to $V$. 
\end{thm}
\begin{proof}
Suppose the sequence $\{T_i\}_{i=0}^\infty$ is bounded, i.e., $\forall i \in \nats: \abs{T_i} \le \Theta$ for some constant $\Theta$. Because $\forall i \in \nats: B_i \ge 2^i$, then $\abs{T_i}/B_i \le \Theta/2^i$ and therefore $\lim_{i \rightarrow \infty} T_i/B_i = 0$.
Now $H_i = V - T_i/B_i$, and so
\begin{align*}
\lim_{i \rightarrow \infty} H_i &= \lim_{i \rightarrow \infty} (V - T_i/B_i) \\
	&= \lim_{i \rightarrow \infty} V -  \lim_{i \rightarrow \infty} T_i/B_i = V .
\end{align*} 
\end{proof}

\section{On-the-fly Technique\label{sec:onthefly}}

When the on-the-fly technique applies, it offers an efficient way to accumulate the (integer) digits generated by a DSM. 
The binary on-the-fly technique can be described as follows.
We assume integers are represented using two's complement notation, and that $\forall i \in \pnats: \beta_i \equiv 2^{\mu_i}$ where each $\mu_i \in \pnats$.

First, no accumulation is needed to form $H_1 = v_1$, nor is there any restriction placed on the magnitude of $v_1$. 
Next, for $i \ge 2$, consider how the digit $v_i$ is accumulated into $H_{i-1}$ to form $H_i$:
\[
H_i \equiv H_{i-1} + \frac{v_i}{B_i} .
\]
Adding $v_i/B_i$ to $H_{i-1}$ creates a carry chain whose length can be nearly the bit-width of $H_{i-1}$.
The goal of the on-the-fly technique is to eliminate this addition and its associated carry chain.

\begin{figure}[t]
	\centering
	\begin{tabular}{crcccc|llll}
		\cmidrule{3-10}    $\beta_i A_{i-1}$ &       & \multicolumn{4}{c|}{$A_{i-1}$} & $0$     & $0$     & $0$     & \multicolumn{1}{c|}{$0$}  \\
		\cmidrule{3-10}    $v_i$ &       &
		\multicolumn{1}{c}{$s$} & \multicolumn{1}{c}{$s$} &  \multicolumn{1}{c|}{$s$} & \multicolumn{1}{c}{$s$} & $v$     & $v$     & $v$     & \multicolumn{1}{c|}{$v$} \\ 
		\cmidrule{3-10}    Sum when $s = 0$ &       & \multicolumn{4}{c|}{$A_{i-1}$} & $v$     & $v$     & $v$     & \multicolumn{1}{c|}{$v$} \\
		\cmidrule{3-10}   Sum when $s = 1$ &       & \multicolumn{4}{c|}{$A_{i-1} - 1$} & $v$     & $v$     & $v$     & \multicolumn{1}{c|}{$v$} \\
		\cmidrule{3-10}    
	\end{tabular}%
	\caption{$1$-bit overlap; $\beta_i A_{i-1} + v_i$ when $\mu_i = 4$ and $s \equiv \text{signbit}(v_i)$.\label{fig:overlap_1b}}
\end{figure}

The simplest form of the on-the-fly technique assumes that 
$\forall i \ge 2: \abs{v_i} < \beta_i$, so both $v_i$ and $ v_i - 1$ have $(\mu_i+1)$-bit two's complement representations.
For each $i \ge 2$,
\[
B_i H_i = B_i H_{i-1} + v_i = \beta_i B_{i-1}H_{i-1} + v_i
\]
and so 
\[
A_i = \beta_i A_{i-1} + v_i
\]
where $A_i \equiv B_i H_i$ is the accumulated value of all of the digits from $v_1$ through $v_i$, inclusively.
Consider Figure \ref{fig:overlap_1b} which illustrates the alignment of $\beta_i A_{i-1}$ and the sign-extended form of $v_i$ when $\mu_i = 4$; note the $1$ bit overlap between the leading (sign) bit of $v_i$ and the trailing bit of $A_{i-1}$. 
When interpreted as a two's complement integer, the value of the bits of the sign-extended form of $v_i$ that overlap $A_{i-1}$ is either $-1$ or $0$. 
From this observation we draw the following conclusions:
\begin{itemize}[nosep]
	\item when $v_i \in \nats$: $s=0$ and $A_i$ is formed by concatenating the bits of $A_{i-1}$ and the $\mu_i$ trailing bits of $v_i$, and
	\item when $v_i < 0$: $s=1$ and $A_i$ is formed by concatenating the bits of $A_{i-1} - 1$ with the $\mu_i$ trailing bits of $v_i$.
\end{itemize}
Consequently, if $A_{i-1}$ and $A'_{i-1} \equiv A_{i-1}-1$ are given, then $A_i$ can be formed by appending the $\mu_i$ trailing bits of $v_i$ to a selection of either $A_{i-1}$ or $A'_{i-1}$. An analogous argument applies to the formation of $A'_i \equiv A_i - 1$ because
\[
A'_i \equiv A_i - 1 = \beta_i A_{i-1} + v_i - 1 = \beta_i A_{i-1} + w_i
\]
where we recall that $w_i \equiv v_i-1$ also has a $(\mu_i+1)$-bit two's complement representation. In summary,
\begin{align*}
A_i &= \begin{cases}
\mbox{concatenate}(A_{i-1},T_{\mu_i}(v_i)) & \text{if}\; v_i \in \nats \\
\mbox{concatenate}(A'_{i-1},T_{\mu_i}(v_i)) & \text{if}\; v_i < 0
\end{cases},  \;\text{and} \\
A'_i &= \begin{cases}
\mbox{concatenate}(A_{i-1},T_{\mu_i}(w_i)) & \text{if}\; w_i \in \nats \\
\mbox{concatenate}(A'_{i-1},T_{\mu_i}(w_i)) & \text{if}\; w_i < 0
\end{cases} .
\end{align*}
where $T_{\mu}(z)$ consist of the trailing $\mu$ bits of the two's complement representation of the integer $z$.

\begin{figure}[t]
	\centering
	\begin{tabular}{crcccc|llll}
		\cmidrule{3-10}    $\beta_i A_{i-1}$ &       & \multicolumn{4}{c|}{$A_{i-1}$} & $0$     & $0$     & $0$     & \multicolumn{1}{c|}{$0$}  \\
		\cmidrule{3-10}    $v_i$ &       &
		\multicolumn{1}{c}{$s$} & \multicolumn{1}{c|}{$s$} &  \multicolumn{1}{c}{$s$} & \multicolumn{1}{c}{$v$} & $v$     & $v$     & $v$     & \multicolumn{1}{c|}{$v$} \\ 
		\cmidrule{3-10}    
	\end{tabular}%
	\caption{$2$-bit overlap; $\beta_i A_{i-1} + v_i$ when $\mu_i = 4$ and $s \equiv \text{signbit}(v_i)$.\label{fig:overlap_2b}}
\end{figure}

This argument can be generalized in several ways. 
Consider, for example, the case where the digits cover the wider range $\forall i \ge 2: \abs{v_i} < 2\beta_i-1$. 
In this case, because $(\mu_i+2)$-bit two's complement integers range from $-2\beta_i$ to $2\beta_i-1$ inclusively, each of the integers $\{v_i-2,v_i-1,v_i,v_i+1\}$ has a $(\mu_i+2)$-bit two's complement representation.
Figure \ref{fig:overlap_2b} illustrates the addition of one of these four integers to $\beta_i A_{i-1}$; note the $2$-bit overlap between that integer and $\beta_i A_{i-1}$. 
The integer described by the bits in the overlap of the sign-extended form of the integer and $\beta_i A_{i-1}$ ranges from $-2$ to $1$, inclusively. 
Therefore, because
\[
A_i + k = \beta_i A_{i-1} + (v_i + k) \quad\text{for}\quad k \in \{-2,-1,0,1\}
\] 
we can form any one of the values $\{A_i-2,A_i-1,A_i,A_i+1\}$ by adding the corresponding integer
$\{v_i-2,v_i-1,v_i,v_i+1\}$ to $\beta_i A_{i-1}$.
For example, to form $A_i-2$ add $z_i \equiv v_i-2$ to $\beta_i A_{i-1}$. To perform this addition use the $2$ leading bits of the $(\mu_i+2)$-bit two's complement representation of $z_i$ to select to which of $\{A_{i-1}-2, A_{i-1}-1, A_{i-1}, A_{i-1}+1\}$ the trailing $\mu_i$-bits of $z_i$ are appended.

\section{DSM Using a Proxy\label{sec:dsm_proxy}}

Algorithm \ref{alg:dsm_basic} is not effective for several reasons.

First, the value of $V$ is used to initialize $T_i$. 
That's acceptable for recoding, where the algorithm converts the value of $V$ in one form (say, binary) into another form (say, decimal). 
It's also acceptable in an analysis of the algorithm. 
It is not acceptable when actually performing a division or square root because it presupposes that the
result of the computation is known before the algorithm starts.

Second, when the algorithm is applied to division or square root, the computation of the tails $T_i$ involves a nontrivial division.
For example, with the invariant written as $\forall i \in \nats: T_i = B_i(V - H_i)$, it is simple to derive for the division problem $V \equiv X/Y$ that
\[
\forall i \in \nats: T_i Y = B_i(X - H_i Y) ,
\]
and for the square root problem $V \equiv \sqrt{X}$ that
\[
\forall i \in \nats: T_i (V+H_i)/2 = B_i(X-H_i^2)/2 .
\]
In each of these equalities, the right-hand side can be computed via addition and multiplication of known finite precision values and the finite precision estimate $H_i$ of $V$.
However, given these right-hand sides, an unavoidable nontrivial division is required to determine the values of $T_i$.

\begin{algorithm}[t]
	\caption{DSM using a proxy that determines $\{(B_i, H_i, T_i)\}_{i=0}^\infty$ for $V \in \nnreals$ where $\forall i \in \pnats: (\dsf_i \in \rni{\Omega_i}) \bigwedge (\beta_i \ge 2)$.\label{alg:dsm_proxy}}
	\begin{algorithmic}[0] 
		\Procedure{DSM\_Proxy}{$V,\{\psi_i\}_{i=0}^{\infty}$}
		\State $(B_0, H_0, T_0) \assign (1, 0, V)$
		\For{$i \assign 0,1,2,\ldots$}
		\State \{\textbf{Invariant}: $V = H_{i} + T_{i}/B_{i}$\}
		\State $T^p_i   \assign (1+\psi_i) T_i$
		\State $v_{i+1} \assign \dsf_{i+1}(\beta_{i+1} T^p_i)$
		\State $B_{i+1} \assign \beta_{i+1} B_{i}$
		\State $H_{i+1} \assign H_{i} + v_{i+1}/B_{i+1}$
		\State $T_{i+1} \assign \beta_{i+1} T_{i} - v_{i+1}$
		\EndFor
		\EndProcedure
	\end{algorithmic}
\end{algorithm} 

Algorithm \ref{alg:dsm_basic} determines the next digit $v_{i+1}$ by approximately rounding $\beta_{i+1} T_i$ to an integer. 
It is plausible, then, that $v_{i+1}$ can be determined using an accurate\footnote{The accuracy of an approximation is measured by its relative error. 
The relative error of an approximation $A'$ of $A \ne 0$ is $\abs{\psi}$ where $A' = (1+\psi) A$.}
proxy $T^p_i$ for $T_i$. 
Algorithm \ref{alg:dsm_proxy} is a template for a DSM that uses a proxy $T^p_i$ for $T_i$; it reduces to Algorithm \ref{alg:dsm_basic} when  $\forall i \in \nats: \psi_i = 0$. 

We make two assumptions about the proxies $\{T^p_i\}_{i=0}^\infty$.
\begin{itemize}[nosep]
\item For analysis: The proxy $T_i^p$ can be expressed as $T^p_i = (1+\psi_i) T_i$; if $T_i \ne 0$ then $\abs{\psi_i}$ is the relative error in the approximation of $T_i$ by the proxy $T^p_i$. 
\item For implementation: The proxy $T_i^p$ can be computed without knowledge of the exact values of $V$ and $T_i$. When this assumption is satisfied, occurrences of $V$ and $T_i$ in Algorithm \ref{alg:dsm_proxy} can be eliminated. 
Examples of this elimination are presented in the following sections.
\end{itemize}

In Algorithm \ref{alg:dsm_proxy}, the sequences $\{\dsf_i\}_{i=1}^\infty$ and $\{\beta_i\}_{i=1}^\infty$ are considered to be fixed and to honor the restrictions stated in the header. 
We also suppose that $\psi_i$ depends on $V$, $T_i$, and $H_i$; the dependence on $H_i$ can be eliminated by applying the invariant $H_i = V - T_i/B_i$. In summary, $T_{i+1}$ can be determined from just $V$ and $T_i$.

To reduce the notational load, the dependence of $T_i$ and $T_i^p$ on $V$ is represented implicitly. 

\begin{thm}
	\label{thm:proxy_theorem} (Proxy Theorem) In Algorithm \ref{alg:dsm_proxy} suppose that for some $V \in \nnreals$ the sequence $\{\psi_i\}_{i=0}^\infty$ satisfies $\forall i\in\nats, t \in \reals: \abs{\psi_i(V,t)} \le \Psi_i(V,\abs{t})$ where $\Psi_i$ is a non-decreasing function of its second argument. Then for that $V$,
    \[
    \forall i \in \nats: \left(\abs{T_i} \le \tau_i(V)\right) \wedge \left(\abs{T^p_i} \le \tau_i^p(V)\right)
    \]
    where $\tau_i, \tau_i^p : \nnreals \rightarrow \reals$ are defined as
	\begin{align*}
    &\tau_0(u) \equiv u, \\
    &\forall i \in \nats: \tau_{i+1}(u) \equiv \beta_{i+1} \Psi_i(u,\tau_i(u)) \tau_i(u) + \Omega_{i+1} ,  \;\text{and} \\[3pt]
    &\forall i \in \nats: \tau_i^p(u) \equiv (1+\Psi_i(u,\tau_i(u))) \tau_i(u) .
   \end{align*}
\end{thm}

\begin{proof}
	Suppose that $V \in \nnreals$ and the sequence $\{\psi_i\}_{i=0}^\infty$ satisfies $\forall i\in\nats, t \in \reals: \abs{\psi_i(V,t)} \le \Psi_i(V,\abs{t})$ where $\Psi_i$ is a non-decreasing function of its second argument.
	
	We inductively prove that $\forall i \in \nats: \abs{T_i} \le \tau_i(V)$ as follows. 
	The base case is true $\abs{T_0} = V = \tau_0(V)$.
	For the inductive step assume that $\abs{T_i} \le \tau_i(V)$ for some $i \in \nats$. 
	We know $T^p_i = (1+\psi_i(V,T_i)) T_i$, so application of the triangle inequality yields:
	\begin{align*}
	\abs{T_{i+1}} 
	&= \abs{\beta_{i+1} T_i - v_{i+1}} \\
	&= \abs{\beta_{i+1} T_i - \dsf_{i+1}(\beta_{i+1} T^p_i)} \\
	&= \abs{\beta_{i+1} T_i - (\beta_{i+1} T^p_i - \codsf_{i+1}(\beta_{i+1} T^p_i))} \\
	&= \abs{\beta_{i+1} (T_i - T^p_i) + \codsf_{i+1}(\beta_{i+1} T^p_i)} \\
	&=\abs{-\beta_{i+1} \psi_i(V,T_i) T_i + \codsf_{i+1}(\beta_{i+1} T^p_i)} \\
	&\le \beta_{i+1} \abs{\psi_i(V,T_i)} \abs{T_i} + \Omega_{i+1} .
	\end{align*}
	Next, apply the assumption that $\abs{\psi_i(V,t)} \le \Psi_i(V,\abs{t})$, where $\Psi_i$ is a non-decreasing function of its second argument, to continue this inequality as follows.
	\begin{align*}
	\abs{T_{i+1}} 
	&\le \beta_{i+1} \abs{\psi_i(V,T_i)} \abs{T_i} + \Omega_{i+1} \\
	&\le \beta_{i+1} \Psi_i(V,\abs{T_i}) \abs{T_i} + \Omega_{i+1} \\
	&\le \beta_{i+1} \Psi_i(V,\tau_i(V)) \tau_i(V) + \Omega_{i+1} \equiv \tau_{i+1} .
	\end{align*}
This completes the induction.

With the bounds on $\forall i \in \nats: \abs{T_i} \le \tau_i(V)$ established, the bounds on $\forall i \in \nats: \abs{T^p_i}$ are obtained as follows. For each $i \in \nats$:
	\begin{align*}
	\abs{T^p_i} &= \abs{(1+\psi_i(V,T_i)) T_i} \\
	            &\le (1+\abs{\psi_i(V,T_i)}) \abs{T_i} \\
                &\le (1+\Psi_i(V,\abs{T_i})) \abs{T_i} \\
	            &\le (1+\Psi_i(V,\tau_i(V))) \tau_i(V) \equiv \tau_i^p(V) .
	\end{align*}
\end{proof}

\begin{defn}
Let $\pee$ be the subset of functions $\preals \rightarrow \reals$ for which
$p \in \pee$ whenever $p(V)$ is a finite sum of terms of the form $c V^n$ where $c \in \nnreals$ and $n \in \ints$. ($\pee$ is a subset of the posynomials in $V$~\cite{boyd2004convex, dufin1967geometric}.)
\end{defn}

Each $p \in \pee$ is a convex function because on $\preals$ its second derivative is non-negative. Among the elements of $\pee$ are each non-negative constant function as well as the identity function $\nu$ where $\forall u \in \preals: \nu(u) = u$. We also have these closure properties: for $p, q \in\pee$ the functions $p / \nu$, $p+q$, $p q \in \pee$.

\begin{cor} \label{cor:proxy_max}
	Let the assumptions of Theorem \ref{thm:proxy_theorem} hold for every $V \in \preals$.
	If 
	\[
	\forall i \in \nats, p \in \pee: \Phi_i(p) \in \pee .
	\]
	where $\Phi_i: \pee\rightarrow\preals\rightarrow\reals$ is defined as 
	\[
	\forall p \in \pee,u \in \preals: \Phi_i(p)(u) = \Psi_i(u,p(u)) .
	\]
	then for any closed subinterval $[a,b]$ of $\preals$,
	\begin{align*}
	&\forall i \in \nats, u \in [a,b]; \tau_i(u) \le t_i \equiv \max{(\tau_i(a),\tau_i(b))},  \\
	&\forall i \in \nats, u \in [a,b]; \tau_i^p(u) \le t^p_i \equiv \max{(\tau_i^p(a),\tau_i^p(b))} .
	\end{align*}
\end{cor}
\begin{proof}
	Let the assumptions of this corollary hold.
	We first prove inductively that $\tau_i \in \pee$ for each $i \in \nats$. 
	The base case is true because $\tau_0 = \nu \in \pee$. For the inductive step let $\tau_i \in \pee$ for some $i \in \nats$. By assumption $\Phi_i(\tau_i) \in \pee$, so by the closure properties $\tau_{i+1} = \beta_{i+1}\Phi_i(\tau_i)\tau_i + \Omega_{i+1} \in \pee$, and this completes the inductive argument. 
	Next, consider $\tau_i^p$ for any $i \in \nats$. 
	By assumption $\Phi_i(\tau_i) \in \pee$ because $\tau_i \in \pee$, so by the closure properties $\tau_i^p = (1+\Phi_i(\tau_i)) \tau_i \in \pee$.
	
	Let $[a,b]$ be a closed subinterval of $\preals$. 
	Because functions in $\pee$ are convex, we know that $\tau_i$ and $\tau_i^p$ attain their maximum on $[a,b]$ at either $a$ or $b$.~\cite{niculescu2006convex}.
\end{proof}

Combining the Theorem \ref{thm:arni_digit_bound} with Corollary \ref{cor:proxy_max} yields for each $i \in \nats$ and $V \in [a,b]$ that
\[
  \abs{T_i} \le t_i\quad\text{and}\quad 
  \abs{v_{i+1}} \le \floor{\beta_{i+1} t^p_i + \Omega_{i+1}} .
\]

The formalization of the Proxy Theorem using the {HOL~Light} theorem prover is presented in the appendix.

\section{DSM for Division\label{sec:dsm_division}}

As discussed in section \ref{sec:scaling}, we consider the computation of $V \equiv X/Y$ where $X\in[1/2,1)$ and $Y\in[1,2)$.
Algorithm \ref{alg:dsm_divide} is an effective DSM that computes $V$; it uses an approximation $g(Y)$ of $1/Y$ obtained from, say, a lookup table. (Microprocessors often have an approximate reciprocal instruction.)
The relative error in this approximation at $Y$ is $\abs{\sigma(Y)}$ where $\sigma : [1,2) \rightarrow \reals$ is defined so that
\[
\forall Y \in [1,2): g(Y) \equiv (1+\sigma(Y))/Y .
\]
We assume $\forall Y \in [1,2): \abs{\sigma(Y)} \le \Sigma$ for some constant $\Sigma$.

\begin{algorithm}[t]
	\caption{DSM using a proxy for division that determines $\{(B_i, H_i, R_i)\}_{i=0}^\infty$ where $X \in [1/2,1)$, $Y \in [1,2)$, $V \equiv X/Y$, and $\forall i \in \pnats: (\dsf_i \in \rni{\Omega_i}) \bigwedge (\beta_i \ge 2)$.\label{alg:dsm_divide}}
	\begin{algorithmic}[0] 
		\Procedure{DSM\_DIV}{$X,Y$}
		\State $(B_0, H_0, R_0) \assign (1, 0, X)$
		\For{$i \assign 0,1,2,\ldots$}
        \State \{\textbf{Invariant}: $X = H_i Y + R_i/B_i$\}
		\State $T^p_i \assign g(Y) R_i$
		\State $v_{i+1} \assign \dsf_{i+1}(\beta_{i+1} T^p_i)$
		\State $B_{i+1} \assign \beta_{i+1} B_{i}$
		\State $H_{i+1} \assign H_{i} + v_{i+1}/B_{i+1}$
		\State $R_{i+1} \assign \beta_{i+1} R_{i} - v_{i+1} Y$
		\EndFor
		\EndProcedure
	\end{algorithmic}
\end{algorithm} 

Reintroduce into Algorithm \ref{alg:dsm_divide} the recursive computation of $T_i$ as in Algorithm \ref{alg:dsm_proxy}, and with it the invariant $\forall i \in \nats: V = H_i + T_i/B_i$. As described in section \ref{sec:dsm_proxy}, from this invariant we find that
\[
\forall i \in \nats: T_i Y = \underbrace{B_i (X - H_i Y)}_{\tilde{R}_i}.
\]
The $\tilde{R}_i$ are called \emph{partial remainders} for division and admit, for all $i \in \nats$, the identity:
\begin{align*} 
\tilde{R}_{i+1} - \beta_{i+1} \tilde{R}_i &= B_{i+1}(X - H_{i+1} Y) - \beta_{i+1} B_i (X - H_i Y) \\
&= - B_{i+1}(H_{i+1} - H_i) Y \\
&= - v_{i+1} Y .
\end{align*}
We conclude that the partial remainders $\tilde{R}_i$ form one solution of the recurrence 
\begin{align*}
&\tilde{R}_0 = X, \\ 
&\forall i \in \nats: 
\tilde{R}_{i+1} = \beta_{i+1} \tilde{R}_i - v_{i+1} Y.
\end{align*}
The $R_i$ computed by Algorithm \ref{alg:dsm_divide} form another solution of this recurrence. Because this recurrence has a unique solution, we conclude that $\forall i \in \nats: \tilde{R}_i = R_i$.

The approximate identity $g(Y)Y \approx 1$ allows division by $Y$ to be replaced, approximately, by multiplication by $g(Y)$. Recall that $\forall i \in \nats: T_i Y = R_i$, so the proxy $T_i^p$ for $T_i$ is
\[
\forall i \in \nats: T^p_i \equiv g(Y) R_i .
\]
A short computation shows that
\begin{align*}
\forall i \in \nats: T^p_i = g(Y) R_i &= g(Y) Y T_i = (1+\sigma(Y)) T_i,
\end{align*}
so the Proxy Theorem \ref{thm:proxy_theorem} can be applied with $\forall i \in \nats: \psi_i(V,t) \equiv \sigma(Y)$ and $\forall i \in \nats:\Psi_i(V,\tau) \equiv \Sigma$ because 
\[
\forall i \in \nats: \abs{\psi_i(V,t)} \equiv \abs{\sigma(Y)} \le \Sigma \equiv \Psi_i(V,\abs{t}) .
\]
Clearly $\forall i \in \nats, p \in \pee: \Phi_i(p) = \Sigma \in \pee$, so Corollary \ref{cor:proxy_max} applies. 
We conclude that $\forall i \in \nats:\abs{T_i} \le t_i \equiv \tau_i(1)$ and $\forall i \in \nats:\abs{T^p_i} \le t_i^p \equiv \tau_i^p(1)$ because
each $\tau_i$ and $\tau_i^p$ is a non-negative increasing linear function on $[1/4,1)$.

\begin{algorithm}[t]
	\caption{DSM using a proxy for square root that determines $\{(B_i, H_i, R_i)\}_{i=0}^\infty$ where $X \in [1/4,1)$, $V \equiv \sqrt{X}$, and $\forall i \in \pnats: (\dsf_i \in \rni{\Omega_i}) \bigwedge (\beta_i \ge 2)$.\label{alg:dsm_squareroot}}
	\begin{algorithmic}[0] 
		\Procedure{DSM\_SQRT}{$X$}
		\State $(B_0, H_0, R_0) \assign (1, 0, X/2)$
		\For{$i \assign 0,1,2,\ldots$}
		\State \{\textbf{Invariant}: $X = H_i^2 + 2R_i/B_i$\}
		\State $T^p_i \assign \mu_i g(X) R_i$
		\State $v_{i+1} \assign \dsf_{i+1}(\beta_{i+1} T^p_i)$
		\State $B_{i+1} \assign \beta_{i+1} B_{i}$
		\State $H_{i+1} \assign H_{i} + v_{i+1}/B_{i+1}$
		\State $R_{i+1} \assign \beta_{i+1} R_{i} - v_{i+1} (H_{i+1} + H_i)/2$
		\EndFor
		\EndProcedure
	\end{algorithmic}
\end{algorithm}


\section{DSM for Square Root\label{sec:dsm_squareroot}}

As discussed in section \ref{sec:scaling}, we consider the computation of $V \equiv \sqrt{X}$ for $X\in[1/4,1)$. Algorithm \ref{alg:dsm_squareroot} is an effective DSM that computes $V$; it uses an approximation $g(X)$ of $1/\sqrt{X}$. (Microprocessors often have an approximate reciprocal square root instruction.)
The relative error in this approximation at $X$ is $\abs{\sigma(X)}$ where $\sigma : [1/4,1) \rightarrow \reals$ is defined so that
\[
\forall X \in [1/4,1): g(X) \equiv (1+\sigma(X))/\sqrt{X} .
\]
We assume $\forall X \in [1/4,1): \abs{\sigma(X)} \le \Sigma$ for some constant $\Sigma$.

Reintroduce into Algorithm \ref{alg:dsm_squareroot} the recursive computation of $T_i$ as in Algorithm \ref{alg:dsm_proxy}, and with it the invariant $\forall i \in \nats: V = H_i + T_i/B_i$. As described in section \ref{sec:dsm_proxy}, from this invariant we find that
\[
\forall i \in \nats: T_i (V+H_i)/2 = \underbrace{B_i (X - H_i^2)/2}_{\tilde{R}_i}.
\]
The $\tilde{R}_i$ are called \emph{partial remainders} for square root and admit, for all $i \in \nats$, the identity:
\begin{align*} 
\tilde{R}_{i+1} - \beta_{i+1} \tilde{R}_i &= B_{i+1}\frac{X - H_{i+1}^2}{2} - \beta_{i+1} B_i \frac{X - H_i^2}{2} \\
&= - B_{i+1}(H_{i+1} - H_i)\frac{H_{i+1} + H_i}{2} \\
&= - v_{i+1}\frac{H_{i+1} + H_i}{2} .
\end{align*}
We conclude that the partial remainders $\tilde{R}_i$ form one solution of the recurrence 
\begin{align*}
&\tilde{R}_0 = X/2, \\ 
&\forall i \in \nats: 
\tilde{R}_{i+1} = \beta_{i+1} \tilde{R}_i - v_{i+1} (H_{i+1}+H_i)/2 .
\end{align*}
The $R_i$ computed by Algorithm \ref{alg:dsm_squareroot} form another solution of this recurrence. Because this recurrence has a unique solution, we conclude that $\forall i \in \nats: \tilde{R}_i = R_i$.

The proxy $T_i^p$ for $T_i$ is obtained by dividing $R_i$ by an approximation of $(V+H_i)/2$.  
We argue that the approximate identity $\forall i \in \nats: \mu_i g(X) (V+H_i)/2 \approx 1$ holds where
\[
\mu_i \equiv (\text{if $i = 0$ then $2$ else $1$})
\] 
because $g(X)V \approx 1$, $H_0 = 0$, and we expect $\forall i \in \pnats: H_i \approx V$. This approximate identity allows division by $(V+H_i)/2$ to be replaced with multiplication by $\mu_i g(X)$, so  the proxy $T_i^p$ for $T_i$ is 
\[
\forall i \in \nats: T^p_i \equiv \mu_i g(X) R_i .
\]
(The invariant tells us that $T_i = 2B_i V$ when $(V+H_i)/2 = 0$.)

Let $X \in [1/4,1)$ be fixed, so $V \equiv \sqrt{X} \in [1/2,1)$.
For any $i \in \nats$ we know $(V+H_i)/2 = V - T_i/(2B_i) = V (1 - T_i/(2VB_i))$ and $g(X)V = 1 + \sigma(X)$, so
\begin{align*}
	T^p_i &\equiv \mu_i g(X)  R_i \\
	      &= \mu_i g(X)          ((V + H_i)/2) T_i \\
	      &= \mu_i g(X) V        (1 - T_i/(2VB_i)) T_i \\
	      &= \mu_i (1+\sigma(X)) (1 - T_i/(2VB_i)) T_i .
\end{align*}
Therefore, $T^p_i = (1+\psi_i(V,T_i)) T_i$ where
\begin{align*}
	\psi_i(V,t) &\equiv \sigma(X) - \begin{cases*}
0 & \text{if}\; $i=0$ \\
(1+\sigma(X))(t/(2 V B_i))& \text{if}\; $i > 0$
\end{cases*}
\end{align*}
because $\mu_0 (1-T_0/(2VB_0)) = 1$, and so the Proxy Theorem \ref{thm:proxy_theorem} can be applied using
\begin{align*}
\Psi_i(V,\abs{t}) &\equiv \Sigma + \begin{cases*}
0 & \text{if}\; $i=0$ \\
(1+\Sigma)(\abs{t}/(2 V B_i))& \text{if}\; $i > 0$
\end{cases*} .
\end{align*}
Note that the first term $\Sigma$ also occurs in $\Psi_i$ for division. Clearly $\forall i \in \nats, p \in \pee: \Phi_i(p) \in \pee$, so
Corollary \ref{cor:proxy_max} applies and we conclude that $\abs{T_i} \le t_i \equiv \max{(\tau_i(1/2),\tau_i(1))}$ and $\abs{T^p_i} \le t_i^p \equiv \max{(\tau_i^p(1/2),\tau_i^p(1))}$.

\section{Application\label{sec:application}}

\begin{sidewaystable}
	\centering
	\caption{DSM (using a proxy) for Division and Square Root with $\Sigma = 2^{-9}$ and $\Omega = 5/8$.\label{tab:dsm_div_sqrt}}%
	\begin{tabular}{ccccccccccccc}
		\hline
		\multicolumn{6}{|c}{Division}                 & \multicolumn{1}{c|}{Digit} & \multicolumn{3}{c|}{V = 1/4} & \multicolumn{3}{c|}{V = 1} \\
		\multicolumn{1}{|c}{i} & $log_2(\beta_i)$ & $\beta_i$ & $B_i$ & $t_i$ & $t^p_i$ & \multicolumn{1}{c|}{Bound} & $\tau_i(V)$ & $\Phi_i(\tau_i)(V)$ & \multicolumn{1}{c|}{$\tau_i^p(V)$} & $\tau_i(V)$ & $\Phi_i(\tau_i)(V)$ & \multicolumn{1}{c|}{$\tau_i^p(V)$} \\
		\hline
		0     &       &       & 1.00E+00 & 1.0000 & 1.0020 &       & 0.2500 & 0.0020 & 0.2505 & 1.0000 & 0.0020 & 1.0020 \\
		1     & 7     & 128   & 1.28E+02 & 0.8750 & 0.8767 & 128   & 0.6875 & 0.0020 & 0.6888 & 0.8750 & 0.0020 & 0.8767 \\
		2     & 7     & 128   & 1.64E+04 & 0.8438 & 0.8454 & 112   & 0.7969 & 0.0020 & 0.7984 & 0.8438 & 0.0020 & 0.8454 \\
		3     & 7     & 128   & 2.10E+06 & 0.8359 & 0.8376 & 108   & 0.8242 & 0.0020 & 0.8258 & 0.8359 & 0.0020 & 0.8376 \\
		4     & 7     & 128   & 2.68E+08 & 0.8340 & 0.8356 & 107   & 0.8311 & 0.0020 & 0.8327 & 0.8340 & 0.0020 & 0.8356 \\
		&       &       &       &       &       &       & & & & & & \\
		\hline
		\multicolumn{6}{|c}{Division}                 & \multicolumn{1}{c|}{Digit} & \multicolumn{3}{c|}{V = 1/4} & \multicolumn{3}{c|}{V = 1} \\
		\multicolumn{1}{|c}{i} & $log_2(\beta_i)$ & $\beta_i$ & $B_i$ & $t_i$ & $t^p_i$ & \multicolumn{1}{c|}{Bound} & $\tau_i(V)$ & $\Phi_i(\tau_i)(V)$ & \multicolumn{1}{c|}{$\tau_i^p(V)$} & $\tau_i(V)$ & $\Phi_i(\tau_i)(V)$ & \multicolumn{1}{c|}{$\tau_i^p(V)$} \\
		\hline
		0     &       &       & 1.00E+00 & 1.0000 & 1.0020 &       & 0.2500 & 0.0020 & 0.2505 & 1.0000 & 0.0020 & 1.0020 \\
		1     & 7     & 128   & 1.28E+02 & 0.8750 & 0.8767 & 128   & 0.6875 & 0.0020 & 0.6888 & 0.8750 & 0.0020 & 0.8767 \\
		2     & 5     & 32    & 4.10E+03 & 0.6797 & 0.6810 & 28    & 0.6680 & 0.0020 & 0.6693 & 0.6797 & 0.0020 & 0.6810 \\
		3     & 7     & 128   & 5.24E+05 & 0.7949 & 0.7965 & 87    & 0.7920 & 0.0020 & 0.7935 & 0.7949 & 0.0020 & 0.7965 \\
		4     & 7     & 128   & 6.71E+07 & 0.8237 & 0.8253 & 102   & 0.8230 & 0.0020 & 0.8246 & 0.8237 & 0.0020 & 0.8253 \\
		&       &       &       &       &       &       &       &       &       &       &       &  \\
		\hline
		\multicolumn{6}{|c}{Square Root}              & \multicolumn{1}{c|}{Digit} & \multicolumn{3}{c|}{V = 1/2} & \multicolumn{3}{c|}{V = 1} \\
		\multicolumn{1}{|c}{i} & $log_2(\beta_i)$ & $\beta_i$ & $B_i$ & $t_i$ & $t^p_i$ & \multicolumn{1}{c|}{Bound} & $\tau_i(V)$ & $\Phi_i(\tau_i)(V)$ & \multicolumn{1}{c|}{$\tau_i^p(V)$} & $\tau_i(V)$ & $\Phi_i(\tau_i)(V)$ & \multicolumn{1}{c|}{$\tau_i^p(V)$} \\
		\hline
		0     &       &       & 1.00E+00 & 1.0000 & 1.0020 &       & 0.5000 & 0.0020 & 0.5010 & 1.0000 & 0.0020 & 1.0020 \\
		1     & 7     & 128   & 1.28E+02 & 0.8750 & 0.8797 & 128   & 0.7500 & 0.0078 & 0.7559 & 0.8750 & 0.0054 & 0.8797 \\
		2     & 7     & 128   & 1.64E+04 & 1.3761 & 1.3789 & 113   & 1.3761 & 0.0020 & 1.3789 & 1.2273 & 0.0020 & 1.2298 \\
		3     & 7     & 128   & 2.10E+06 & 0.9838 & 0.9858 & 177   & 0.9838 & 0.0020 & 0.9858 & 0.9377 & 0.0020 & 0.9396 \\
		4     & 7     & 128   & 2.68E+08 & 0.8710 & 0.8727 & 126   & 0.8710 & 0.0020 & 0.8727 & 0.8595 & 0.0020 & 0.8611 \\
		&       &       &       &       &       &       & & & & & & \\
		\hline
		\multicolumn{6}{|c}{Square Root}              & \multicolumn{1}{c|}{Digit} & \multicolumn{3}{c|}{V = 1/2} & \multicolumn{3}{c|}{V = 1} \\
		\multicolumn{1}{|c}{i} & $log_2(\beta_i)$ & $\beta_i$ & $B_i$ & $t_i$ & $t^p_i$ & \multicolumn{1}{c|}{Bound} & $\tau_i(V)$ & $\Phi_i(\tau_i)(V)$ & \multicolumn{1}{c|}{$\tau_i^p(V)$} & $\tau_i(V)$ & $\Phi_i(\tau_i)(V)$ & \multicolumn{1}{c|}{$\tau_i^p(V)$} \\
		\hline
		0     &       &       & 1.00E+00 & 1.0000 & 1.0020 &       & 0.5000 & 0.0020 & 0.5010 & 1.0000 & 0.0020 & 1.0020 \\
		1     & 7     & 128   & 1.28E+02 & 0.8750 & 0.8797 & 128   & 0.7500 & 0.0078 & 0.7559 & 0.8750 & 0.0054 & 0.8797 \\
		2     & 5     & 32    & 4.10E+03 & 0.8128 & 0.8145 & 28    & 0.8128 & 0.0022 & 0.8145 & 0.7756 & 0.0020 & 0.7772 \\
		3     & 7     & 128   & 5.24E+05 & 0.8489 & 0.8505 & 104   & 0.8489 & 0.0020 & 0.8505 & 0.8283 & 0.0020 & 0.8299 \\
		4     & 7     & 128   & 6.71E+07 & 0.8374 & 0.8390 & 109   & 0.8374 & 0.0020 & 0.8390 & 0.8322 & 0.0020 & 0.8338 \\
	\end{tabular}%
\end{sidewaystable}

The results displayed in Table~\ref{tab:dsm_div_sqrt} describe the evolution of the bounds on the tails, tail proxies, and digits for the DSM algorithms for division and square root presented in the previous two sections. 
In this table the reciprocal and reciprocal root approximations are characterized by $\Sigma \equiv 2^{-9}$, and all digit selection functions belong to $\rni{\Omega}$ for $\Omega \equiv 5/8$.
(The $P N^2$ or $P N Q$ recoders discussed in~\cite{daumas2003further} provide such digit selection functions.)

The table displays results for two choices of $\beta$-sequence:
\begin{itemize}[nosep]
	\item $\{\beta_1,\beta_2, \beta_3, \beta_4\} \equiv \{2^7,2^7,2^7,2^7\}$, and 
	\item $\{\beta_1,\beta_2, \beta_3, \beta_4\} \equiv \{2^7,2^5,2^7,2^7\}$.
\end{itemize}
for each of division and square root. For each of these we obtain from Corollary \ref{cor:proxy_max}, with $\nu$ the identity function, that
\begin{align*}
&\tau_0 \equiv \nu, \\
&\forall i \in \nats: \tau_{i+1} \equiv \beta_{i+1} \Phi_i(\tau_i) \tau_i + \Omega ,  \\[3pt]
&\forall i \in \nats: \tau_i^p \equiv (1+\Phi_i(\tau_i)) \tau_i 
\end{align*}
where for division
\[
\forall \tau \in \pee: \Phi_i(\tau) \equiv \Sigma
\]
while for square root
\[
\forall \tau \in \pee: \Phi_i(\tau) \equiv 
\begin{cases*}
\Sigma & \text{if}\; $i=0$ \\
\Sigma + (1+\Sigma)\tau/(2\nu B_i) & \text{if}\; $i > 0$
\end{cases*} .
\]
For any given value of $V$, we know the value of $\tau_0$ and so we can compute $\Phi_0(\tau_0)(V)$ and then $\tau^p_0(V)$. 
This pattern is repeated for $i=1,2,3,4$ in succession; compute $\tau_i(V)$, then $\Phi_i(\tau_i)(V)$ and $\tau_i^p(V)$. 
From Corollary \ref{cor:proxy_max} we obtain
\begin{align*}
&\forall i \in \nats, V \in [a,b]; \tau_i(V) \le t_i \equiv \max{(\tau_i(a),\tau_i(b))},  \;\text{and}\\
&\forall i \in \nats, V \in [a,b]; \tau_i^p(V) \le t^p_i \equiv \max{(\tau_i^p(a),\tau_i^p(b))}
\end{align*}
where $[a,b] \equiv [1/4,1]$ for division and $[a,b] \equiv [1/2,1]$ for square root. Finally, for $i = 1,2,3,4$:
\[
\abs{T_i} \le t_i, \quad\text{and}\quad
\abs{v_i} \le \floor{\beta_i t^p_{i-1} + \Omega} .
\]

Observe that, for square root, the first $\beta$-sequence leads to an upper bound on $\abs{T_2}$ that is larger than $1$, and so the bound on $\abs{v_3}$ is larger than $2^{\beta_3} = 2^7 = 128$. 
For the second $\beta$-sequence, obtained from the first $\beta$-sequence by decreasing $\beta_2$ from $2^7$ to $2^5$, we find that $\abs{v_i} < 2^{\beta_i}$ for $2 \le i \le 4$ as well as $\abs{T_4} < 1$; so the simplest form of on-the-fly accumulation of the digits can be applied. 
The reason why the reduction of $\beta_2$ from $2^7$ to $2^5$ is effective can be explained by the fact that
\[
\Phi_1(\tau_1) = \Sigma + (1+\Sigma)\tau_1/(2\nu\beta_1)
\]
and so
\begin{align*}
\tau_2  &= \beta_2\Sigma + \Omega_2 + \beta_2(1+\Sigma)\tau_1/(2\nu\beta_1) .
\end{align*}
From the corresponding example for division we know $\beta_2 \Sigma + \Omega_2 = 1/4 + 5/8 = 7/8$ when $\beta_2 = 2^7$.
The third term contains the ratio $\beta_2/\beta_1$, so when $\beta_2$ is reduced from $2^7$ to $2^5$ the contribution of this third term is reduced by a factor of $4$. 

We performed additional experiments using a spreadsheet implementation of the DSM for division and square root~\footnote{For readers interested in replicting our results: These Excel 2016 spreadsheets are included as ancillary files {\tt DSM\_Division.xlsm} and {\tt DSM\_SquareRoot.xlsm}. The definition of the functions {\tt dsf()}, {\tt phidiv()}, and {\tt phisqrt()} used in these spreadsheets are contained in a VBA Module. The optimization was performed by Excel's Solver Add-in using its Evolutionary mode of operation.} that expand on the results presented in Table \ref{tab:dsm_div_sqrt}. 
For specified values of the inputs ($X$ and $Y$ for division, $X$ for square root), the spreadsheet computed the slack $s_i \equiv v^{max}_i - \abs{v_i}$ where $v^{max}_i$ is the upper bound on $\abs{v_i}$ as discussed at the end of section \ref{sec:dsm_proxy}.
The spreadsheet's optimizer was used to determine inputs that made $s_i$ small, i.e., made $\abs{v_i}$ close to $v^{max}_i$. 
For both division and square root, and for each $i \in \{1,2,3,4\}$, the optimizer was able to find inputs that made $\abs{v_i}$ at least $96$ percent of $v^{max}_i$.

\section{Conclusion\label{sec:conclusion}}

The analysis presented in this paper is generic in the sense that no special properties of digit selection or reciprocal approximation are assumed.
We have not considered how the digit selection function is implemented efficiently; we refer only to the references~\cite{daumas1997recoders, daumas2003further, ercegovac1994recoding,  ercegovac1994very, lang1995very}. 
Nor have we discussed the effect of using one-sided approximations of the reciprocals, or biased digit selection functions. 

The analysis presented here also extends to higher roots. 
For example, for the cube root $V = X^{1/3}$, from $T_i = B_i(V - H_i)$ it follows that 
\[
T_i (V^2 + VH_i + H_i^2)/3 = B_i(X - H_i^3)/3 .
\]
The partial remainders $R_i \equiv B_i(X - H_i^3)/3$ satisfy a two-term recurrence. 
Also, if $\nu_i = (\text{if i == 0 then 3 else 1})$ and $g(X) \approx X^{-2/3}$, then $T_i^p \equiv \nu_i g(X) R_i$ is a natural choice as the proxy for $T_i$ because $\nu_ig(X)(V^2 + VH_i + H_i^2)/3 \approx 1$.

Prescaled division is also covered by the analysis presented here. 
Prescaled division computes $X' \equiv g(Y) X$ and $Y' \equiv g(Y ) Y = 1 + \sigma(Y)$ before the for-loop; note that $X'/Y' = X/Y$.
Inside the for-loop, the expressions
\begin{align*}
R_{i+1} &= \beta_{i+1} R_i - v_{i+1} Y  \quad\text{and}\quad
X = H_i Y + R_i/B_i
\end{align*}
for the partial remainder and the invariant become, after multiplication by $g(Y)$,
\begin{align*}
R'_{i+1} &= \beta_{i+1} R'_i - v_{i+1}  Y' \\
&= (\beta_{i+1} R'_i - v_{i+1}) - v_{i+1} \sigma(Y), \;\text{and} \\
X' &= H_i Y' + R'_i/B_i
\end{align*}
where $R'_i \equiv g(Y) R_i$. Note that $R'_0 \equiv X'$ and $T^p_i = R'_i$. The advantage of prescaled division is that, at a cost of two multiplications outside the for-loop, no multiplication inside the for-loop is needed to form the proxy $T^p_i$.  

The proofs of the Proxy Theorem, its Corollary and the applications to division and square root, including verification of some concrete error bounds for particular instances, have been formally verified using the HOL Light theorem prover~\cite{harrison1996hol}; for the details see the appendix.


\section*{Acknowledgments\label{sec:acknowledgments}}

We thank Ping Tak Peter Tang, John O'Leary, Simon Rubanovich, and David Russinoff for many conversations related to the design and validation of floating-point arithmetic units.


\appendix

\section{HOL Light proof of theorems}\label{section:appendix}

In this appendix, we discuss the full HOL Light~\cite{harrison1996hol} proof script
for the claims made in the main body of the paper.

\subsection{The main theorem~\ref{thm:proxy_theorem}}

From this point on we present the actual ASCII proof script\footnote{This HOL Light script is included as the ancillary file dsm.ml.} required for HOL
Light to prove the statements, interspersed with a few comments. Initially we
load HOL Light's fairly extensive library of multivariate real and complex
analysis. This is overkill for the relatively small amount of background
material we need, but saves us from establishing from scratch various basic
properties of convex functions. (In fact, a couple of additional properties of
convex functions of general interest were added to the libraries as a direct
result of supporting this proof.)

\begin{scriptsize}\begin{verbatim}
needs "Multivariate/realanalysis.ml";;
\end{verbatim}\end{scriptsize}

We now proceed to the main proof scripts. Note that HOL Light proof scripts
are normally wrapped up in a {\verb|prove(assertion,tactics)|} pair, but that
the intermediate steps can be explored interactively via commands such as {\tt
g} (set goal) and {\tt e} (expand current goal using tactics). For more
information about the mechanics of HOL Light interaction see the tutorial
\cite{harrison-tutorial}. Thus, the overall block for theorem~\ref{thm:proxy_theorem} is an OCaml
phrase binding to the desired name {\verb|THEOREM_V_1|} the result of proving
an assertion

\begin{scriptsize}\begin{verbatim}
let THEOREM_V_1 = prove
 (`!(V:real) (beta:num->real) (omega:num->real) (DSF:num->real->real)
    (B:num->real) (H:num->real) (v:num->real) (Tl:num->real) (Tp:num->real)
    (PSI:num->real#real->real) (psi:num->real#real->real)
    (tau:num->real->real) (taup:num->real->real).

        // Environmental assumptions including nondecreasing property
        &0 <= V /\
        (!i. i >= 0 ==> beta i > &0) /\
        (!i. i >= 1 ==> (!x. abs (x - DSF i x) <= omega i)) /\
        (!i. i >= 0 ==> abs (psi i (V,Tl i)) <= PSI i (V,abs(Tl i))) /\
        (!i x y. &0 <= x /\ x <= y ==> PSI i (V,x) <= PSI i (V,y)) /\

        (!u. tau 0 u = u) /\
        (!i u. tau (i + 1) u =
               beta (i + 1) * PSI i (u,tau i u) * tau i u + omega (i + 1)) /\
        (!i u. taup i u = (&1 + PSI i (u,tau i u)) * tau i u) /\

        // Computing recursively
        B 0 = &1 /\ H 0 = &0 /\ Tl 0 = V /\

        (!i. Tp i = (&1 + psi i (V,Tl i)) * Tl i) /\
        (!i. v (i + 1) = DSF (i + 1) (beta (i + 1) * Tp i)) /\
        (!i. B (i + 1) = beta (i + 1) * B i) /\
        (!i. H (i + 1) = H i + v (i + 1) / B (i + 1)) /\
        (!i. Tl (i + 1) = beta (i + 1) * Tl i - v (i + 1))

        // Conclude loop invariant and bounds.
        ==> (!i. V = H i + Tl i / B i) /\
            (!i. i >= 0 ==> abs(Tl i) <= tau i V) /\
            (!i. i >= 0 ==> abs(Tp i) <= taup i V)`,
\end{verbatim}\end{scriptsize}

\noindent using the tactic script that follows, starting with some initial
breakdown of the goal stripping off outer quantifiers and turning the
antecedents of implications into assumptions of the goal state:

\begin{scriptsize}\begin{verbatim}
  REPEAT GEN_TAC THEN REWRITE_TAC[GE; real_gt; real_gt; LE_0] THEN
  STRIP_TAC THEN
\end{verbatim}\end{scriptsize}

We first establish by induction that all $B_i$ are strictly positive:

\begin{scriptsize}\begin{verbatim}
  SUBGOAL_THEN `!i:num. &0 < B i` ASSUME_TAC THENL
   [INDUCT_TAC THEN ASM_SIMP_TAC[REAL_LT_01; ADD1; REAL_LT_MUL];
    ALL_TAC] THEN
\end{verbatim}\end{scriptsize}

We then reshuffle the conjuncts to handle the $\tau^p$ clause first, assuming
the other two clauses:

\begin{scriptsize}\begin{verbatim}
  MATCH_MP_TAC(TAUT `(p /\ q ==> r) /\ p /\ q ==> p /\ q /\ r`) THEN
  CONJ_TAC THENL
   [DISCH_THEN(STRIP_ASSUME_TAC o GSYM) THEN
    ASM_REWRITE_TAC[REAL_ABS_MUL] THEN GEN_TAC THEN
    MATCH_MP_TAC REAL_LE_MUL2 THEN ASM_REWRITE_TAC[REAL_ABS_POS] THEN
    MATCH_MP_TAC(REAL_ARITH `abs(x) <= a ==> abs(&1 + x) <= &1 + a`) THEN
    TRANS_TAC REAL_LE_TRANS `(PSI:num->real#real->real) i (V,abs(Tl i))` THEN
    ASM_SIMP_TAC[] THEN ASM_MESON_TAC[REAL_ABS_POS];
    ALL_TAC] THEN
\end{verbatim}\end{scriptsize}

Now we begin the main inductive proof and dispose of the base case by simple
arithmetic:

\begin{scriptsize}\begin{verbatim}
  REWRITE_TAC[AND_FORALL_THM] THEN
  INDUCT_TAC THEN ASM_REWRITE_TAC[] THENL [ASM_REAL_ARITH_TAC; ALL_TAC] THEN
\end{verbatim}\end{scriptsize}

First we establish that the step case of the loop invariant holds

\begin{scriptsize}\begin{verbatim}
  CONJ_TAC THENL
   [ASM_REWRITE_TAC[ADD1] THEN
    SUBGOAL_THEN `&0 < beta (i + 1) /\ &0 < B i` MP_TAC THENL
     [ASM_REWRITE_TAC[]; CONV_TAC REAL_FIELD];
    ALL_TAC] THEN
\end{verbatim}\end{scriptsize}

\noindent after which we massage the goal a little and chain through the
inequalities, roughly following the paper proof:

\begin{scriptsize}\begin{verbatim}
  FIRST_X_ASSUM(CONJUNCTS_THEN (ASSUME_TAC o GSYM)) THEN
  REWRITE_TAC[ADD1] THEN

  TRANS_TAC REAL_LE_TRANS
   `abs(-- beta (i + 1)  * psi i (V:real,Tl i) * Tl i +
        (beta (i + 1) * Tp i - DSF (i + 1) (beta (i + 1) * Tp i)))` THEN
  CONJ_TAC THENL [ASM_REWRITE_TAC[] THEN REAL_ARITH_TAC; ALL_TAC] THEN

  TRANS_TAC REAL_LE_TRANS
   `beta (i + 1) * abs(psi i (V:real,Tl i)) * abs(Tl i) + omega(i + 1)` THEN
  CONJ_TAC THENL
   [MATCH_MP_TAC(REAL_ARITH
     `abs(x) <= a /\ abs(y) <= b ==> abs(x + y) <= a + b`) THEN
    ASM_SIMP_TAC[ARITH_RULE `1 <= i + 1`] THEN
    REWRITE_TAC[REAL_ABS_MUL; REAL_ABS_NEG] THEN
    ASM_SIMP_TAC[REAL_ARITH `&0 < x ==> abs x = x`; REAL_LE_REFL];
    ALL_TAC] THEN

  ASM_REWRITE_TAC[] THEN REWRITE_TAC[REAL_LE_RADD] THEN
  ASM_SIMP_TAC[REAL_LE_LMUL_EQ] THEN
  MATCH_MP_TAC REAL_LE_MUL2 THEN ASM_REWRITE_TAC[REAL_ABS_POS] THEN
  TRANS_TAC REAL_LE_TRANS `(PSI:num->real#real->real) i (V,abs(Tl i))` THEN
  ASM_SIMP_TAC[] THEN ASM_MESON_TAC[REAL_ABS_POS]);;
\end{verbatim}\end{scriptsize}

\subsection{Properties of posynomials}

The proof of corollary~\ref{cor:proxy_max} requires a notion corresponding to a restricted subset of the
posynomials, functions of $V$ consisting of finite sums of positive multiples
of integer powers of $V$, $\sum_1^k c_i V^{r_i}$. We render this in HOL Light
as follows (using the simple word `posynomial' is perhaps a little misleading
since these are a restricted case, but this is only a name):

\begin{scriptsize}\begin{verbatim}
let posynomial = new_definition
 `posynomial p <=>
  ?c (n:num->real) k.
        (!i. 1 <= i /\ i <= k ==> c i > &0 /\ integer(n i)) /\
        (!v. &0 < v ==> sum (1..k) (\i. c i * v rpow (n i)) = p v)`;;
\end{verbatim}\end{scriptsize}

We now proceed to prove various basic `closure' properties, roughly
corresponding to those mentioned in the text. The identically zero function is
a posynomial; even though the coefficients in the sum are assumed strictly
positive, we can take $k = 0$ and get an empty sum:

\begin{scriptsize}\begin{verbatim}
let POSYNOMIAL_0 = prove
 (`posynomial (\v. &0)`,
  REWRITE_TAC[posynomial] THEN
  MAP_EVERY EXISTS_TAC [`(\i. &1):num->real`; `(\i. &0):num->real`; `0`] THEN
  REWRITE_TAC[SUM_CLAUSES_NUMSEG] THEN ARITH_TAC);;
\end{verbatim}\end{scriptsize}

Similarly straightforwardly, the identically 1 function is also a posynomial:

\begin{scriptsize}\begin{verbatim}
let POSYNOMIAL_1 = prove
 (`posynomial (\v. &1)`,
  REWRITE_TAC[posynomial] THEN
  MAP_EVERY EXISTS_TAC [`(\i. &1):num->real`; `(\i. &0):num->real`; `1`] THEN
  REWRITE_TAC[INTEGER_CLOSED; SUM_SING_NUMSEG; RPOW_POW] THEN REAL_ARITH_TAC);;
\end{verbatim}\end{scriptsize}

\noindent and indeed if $p$ is a posynomial, so is any positive multiple of it

\begin{scriptsize}\begin{verbatim}
let POSYNOMIAL_CMUL = prove
 (`!p c. posynomial p /\ &0 < c ==> posynomial(\v. c * p(v))`,
  REPEAT GEN_TAC THEN
  DISCH_THEN(CONJUNCTS_THEN2 MP_TAC ASSUME_TAC) THEN
  REWRITE_TAC[posynomial] THEN DISCH_THEN(X_CHOOSE_THEN `d:num->real`
   (fun th -> EXISTS_TAC `(\i. c * d i):num->real` THEN MP_TAC th)) THEN
  REPEAT(MATCH_MP_TAC MONO_EXISTS THEN GEN_TAC) THEN
  SIMP_TAC[SUM_LMUL; GSYM REAL_MUL_ASSOC] THEN
  ASM_SIMP_TAC[real_gt; REAL_LT_MUL]);;
\end{verbatim}\end{scriptsize}

It is in fact convenient to record that any nonnegative constant function is a
posynomial

\begin{scriptsize}\begin{verbatim}
let POSYNOMIAL_CONST = prove
 (`!c. &0 <= c ==> posynomial (\v. c)`,
  REWRITE_TAC[REAL_ARITH `&0 <= c <=> c = &0 \/ &0 < c`] THEN
  REPEAT STRIP_TAC THEN ASM_REWRITE_TAC[POSYNOMIAL_0] THEN
  GEN_REWRITE_TAC (RAND_CONV o ABS_CONV) [GSYM REAL_MUL_RID] THEN
  MATCH_MP_TAC POSYNOMIAL_CMUL THEN
  ASM_REWRITE_TAC[POSYNOMIAL_1]);;
\end{verbatim}\end{scriptsize}

We next observe that multiplying a posynomial by an integer power of the
variable again gives a posynomial:

\begin{scriptsize}\begin{verbatim}
let POSYNOMIAL_VPOWMUL = prove
 (`!p n. posynomial p /\ integer n ==> posynomial(\v. p(v) * v rpow n)`,
  REPEAT GEN_TAC THEN DISCH_THEN(CONJUNCTS_THEN2 MP_TAC ASSUME_TAC) THEN
  REWRITE_TAC[posynomial] THEN
  MATCH_MP_TAC MONO_EXISTS THEN X_GEN_TAC `c:num->real` THEN
  GEN_REWRITE_TAC BINOP_CONV [SWAP_EXISTS_THM] THEN
  MATCH_MP_TAC MONO_EXISTS THEN X_GEN_TAC `k:num` THEN
  DISCH_THEN(X_CHOOSE_THEN `nn:num->real` STRIP_ASSUME_TAC) THEN
  EXISTS_TAC `(\i. nn i + n):num->real` THEN
  ASM_SIMP_TAC[RPOW_ADD; REAL_MUL_ASSOC; SUM_RMUL; INTEGER_CLOSED]);;
\end{verbatim}\end{scriptsize}

This yields other basic closure properties as special cases: multiplying by $V$
and dividing by $V$:

\begin{scriptsize}\begin{verbatim}
let POSYNOMIAL_VMUL = prove
 (`!p. posynomial p ==> posynomial(\v. p(v) * v)`,
  REPEAT STRIP_TAC THEN
  MP_TAC(ISPECL [`p:real->real`; `&1:real`] POSYNOMIAL_VPOWMUL) THEN
  ASM_REWRITE_TAC[RPOW_POW; REAL_POW_1; INTEGER_CLOSED]);;

let POSYNOMIAL_VDIV = prove
 (`!p. posynomial p ==> posynomial(\v. p(v) / v)`,
  REPEAT STRIP_TAC THEN
  MP_TAC(ISPECL [`p:real->real`; `-- &1:real`] POSYNOMIAL_VPOWMUL) THEN
  ASM_SIMP_TAC[RPOW_POW; real_div; RPOW_NEG; REAL_POW_1; INTEGER_CLOSED]);;
\end{verbatim}\end{scriptsize}

We can also trivially derive that the identity function is a posynomial:

\begin{scriptsize}\begin{verbatim}
let POSYNOMIAL_V = prove
 (`posynomial(\v. v)`,
  GEN_REWRITE_TAC (RAND_CONV o ABS_CONV) [GSYM REAL_MUL_LID] THEN
  MATCH_MP_TAC POSYNOMIAL_VMUL THEN REWRITE_TAC[POSYNOMIAL_1]);;
\end{verbatim}\end{scriptsize}

Slightly more involved is the fact that the sum of posynomials is a posynomial;
note that following the strict form of the definition we need to plug two
summations $1 \ldots n_1$ and $1 \ldots n_2$ into a single summation $1 \ldots
n_1 + n_2$ with some straightforward but fiddly reindexing:

\begin{scriptsize}\begin{verbatim}
let POSYNOMIAL_ADD = prove
 (`!p q. posynomial p /\ posynomial q ==> posynomial(\v. p v + q v)`,
  REPEAT GEN_TAC THEN
  REWRITE_TAC[posynomial; IMP_CONJ; LEFT_IMP_EXISTS_THM] THEN
  MAP_EVERY X_GEN_TAC [`c1:num->real`; `n1:num->real`; `m:num`] THEN
  DISCH_TAC THEN DISCH_TAC THEN
  MAP_EVERY X_GEN_TAC [`c2:num->real`; `n2:num->real`; `n:num`] THEN
  DISCH_TAC THEN DISCH_TAC THEN
  EXISTS_TAC `\i. if i <= m then (c1:num->real) i else c2 (i - m)` THEN
  EXISTS_TAC `\i. if i <= m then (n1:num->real) i else n2 (i - m)` THEN
  EXISTS_TAC `m + n:num` THEN REWRITE_TAC[] THEN CONJ_TAC THENL
   [REPEAT STRIP_TAC THEN COND_CASES_TAC THEN ASM_SIMP_TAC[] THEN
    ASM_MESON_TAC[ARITH_RULE
     `~(i:num <= m) /\ i <= m + n ==> 1 <= i - m /\ i - m <= n`];
    REPEAT STRIP_TAC THEN ONCE_REWRITE_TAC[COND_RAND] THEN
    ONCE_REWRITE_TAC[MESON[] `(if p then f else g) (if p then x else y) =
        if p then f x else g y`] THEN
    SIMP_TAC[SUM_CASES; FINITE_NUMSEG; IN_NUMSEG;
      ARITH_RULE `(1 <= i /\ i <= m + n) /\ i <= m <=> 1 <= i /\ i <= m`;
      ARITH_RULE `(1 <= i /\ i <= m + n) /\ ~(i <= m) <=>
                  1 + m <= i /\ i <= n + m`] THEN
    REWRITE_TAC[GSYM numseg; SUM_OFFSET; ADD_SUB] THEN ASM_SIMP_TAC[]]);;
\end{verbatim}\end{scriptsize}

Now by induction we can establish that a finite sum of posynomials (based on
some arbitrary indexing set $k$) is a posynomial:

\begin{scriptsize}\begin{verbatim}
let POSYNOMIAL_SUM = prove
 (`!k:A->bool p.
        FINITE k /\ (!i. i IN k ==> posynomial(\v. p v i))
        ==> posynomial (\v. sum k (p v))`,
  REWRITE_TAC[IMP_CONJ; RIGHT_FORALL_IMP_THM] THEN
  MATCH_MP_TAC FINITE_INDUCT_STRONG THEN
  SIMP_TAC[SUM_CLAUSES; POSYNOMIAL_0; POSYNOMIAL_ADD; FORALL_IN_INSERT;
            ETA_AX]);;
\end{verbatim}\end{scriptsize}

This yields without too much trouble the fact that the product of posynomials
is a posynomial, simply by expanding the product of sums into a single sum over
the Cartesian product of the indexing set (using HOL Light's standard theorem
{\verb|SUM_SUM_PRODUCT|}) and appealing to the just-proved
{\verb|POSYNOMIAL_SUM|}:

\begin{scriptsize}\begin{verbatim}
let POSYNOMIAL_MUL = prove
 (`!p q. posynomial p /\ posynomial q ==> posynomial(\v. p v * q v)`,
  REPEAT GEN_TAC THEN GEN_REWRITE_TAC (LAND_CONV o BINOP_CONV)
   [CONV_RULE (RAND_CONV(ONCE_DEPTH_CONV SYM_CONV)) (SPEC_ALL posynomial)] THEN
  STRIP_TAC THEN ASM_SIMP_TAC[posynomial] THEN
  REWRITE_TAC[GSYM posynomial] THEN
  SIMP_TAC[SUM_SUM_PRODUCT; FINITE_NUMSEG; REAL_MUL_SUM] THEN
  MATCH_MP_TAC POSYNOMIAL_SUM THEN
  SIMP_TAC[FINITE_PRODUCT_DEPENDENT; FINITE_NUMSEG; FORALL_IN_GSPEC] THEN
  REWRITE_TAC[IN_NUMSEG] THEN REPEAT STRIP_TAC THEN
  ONCE_REWRITE_TAC[REAL_ARITH
   `(c * x) * (d * y):real = (c * d) * (x * y)`] THEN
  SIMP_TAC[posynomial; GSYM RPOW_ADD] THEN REWRITE_TAC[GSYM posynomial] THEN
  MATCH_MP_TAC POSYNOMIAL_VPOWMUL THEN ASM_SIMP_TAC[INTEGER_CLOSED] THEN
  ONCE_REWRITE_TAC[GSYM REAL_MUL_RID] THEN
  RULE_ASSUM_TAC(REWRITE_RULE[real_gt]) THEN
  MATCH_MP_TAC POSYNOMIAL_CMUL THEN
  ASM_SIMP_TAC[REAL_LT_MUL; POSYNOMIAL_1]);;
\end{verbatim}\end{scriptsize}

Finally, we prove that each posynomial defines a convex function on the
positive reals. (For more on convex functions see any standard book on
convexity, e.g. \cite{barvinok-convexity} or \cite{webster-convexity}.)

\begin{scriptsize}\begin{verbatim}
let REAL_CONVEX_ON_POSYNOMIAL = prove
 (`!p. posynomial p ==> p real_convex_on {x | x > &0}`,
  GEN_TAC THEN REWRITE_TAC[posynomial; LEFT_IMP_EXISTS_THM; real_gt] THEN
  MAP_EVERY X_GEN_TAC [`c:num->real`; `n:num->real`; `m:num`] THEN
  DISCH_THEN(CONJUNCTS_THEN2 ASSUME_TAC MP_TAC) THEN
  GEN_REWRITE_TAC (LAND_CONV o ONCE_DEPTH_CONV)
   [SET_RULE `&0 < v <=> v IN {x | &0 < x}`] THEN
  MATCH_MP_TAC(MESON[REAL_CONVEX_ON_EQ]
   `is_realinterval s /\ f real_convex_on s
    ==> (!x. x IN s ==> f x = g x) ==> g real_convex_on s`) THEN
  REWRITE_TAC[IS_REALINTERVAL_CLAUSES] THEN
  MATCH_MP_TAC REAL_CONVEX_ON_SUM THEN
  REWRITE_TAC[FINITE_NUMSEG; IN_NUMSEG] THEN
  X_GEN_TAC `i:num` THEN STRIP_TAC THEN MATCH_MP_TAC REAL_CONVEX_LMUL THEN
  ASM_SIMP_TAC[REAL_LT_IMP_LE] THEN
  MATCH_MP_TAC REAL_CONVEX_ON_RPOW_INTEGER THEN
  ASM SET_TAC[]);;
\end{verbatim}\end{scriptsize}

\subsection{Corollary~\ref{cor:proxy_max}}

We can now establish the corollary:

\begin{scriptsize}\begin{verbatim}
let COROLLARY_V_3 = prove
 (`!(V:real) (beta:num->real) (omega:num->real) (DSF:num->real->real)
    (B:num->real) (H:num->real) (v:num->real) (Tl:num->real) (Tp:num->real)
    (PSI:num->real#real->real) (psi:num->real#real->real)
    (tau:num->real->real) (taup:num->real->real).

        // Environmental assumptions including nondecreasing property
        &0 < V /\
        (!i. i >= 0 ==> beta i > &0) /\
        (!i. i >= 1 ==> (!x. abs (x - DSF i x) <= omega i)) /\
        (!i. i >= 0 ==> abs (psi i (V,Tl i)) <= PSI i (V,abs(Tl i))) /\
        (!i x y. &0 <= x /\ x <= y ==> PSI i (V,x) <= PSI i (V,y)) /\

        (!u. tau 0 u = u) /\
        (!i u. tau (i + 1) u =
               beta (i + 1) * PSI i (u,tau i u) * tau i u + omega (i + 1)) /\
        (!i u. taup i u = (&1 + PSI i (u,tau i u)) * tau i u) /\

        // Computing recursively
        B 0 = &1 /\ H 0 = &0 /\ Tl 0 = V /\

        (!i. Tp i = (&1 + psi i (V,Tl i)) * Tl i) /\
        (!i. v (i + 1) = DSF (i + 1) (beta (i + 1) * Tp i)) /\
        (!i. B (i + 1) = beta (i + 1) * B i) /\
        (!i. H (i + 1) = H i + v (i + 1) / B (i + 1)) /\
        (!i. Tl (i + 1) = beta (i + 1) * Tl i - v (i + 1)) /\

        // The extra posynomial-related assumption
        (!i p. i >= 0 /\ posynomial p
               ==> posynomial (\v. PSI i (v,p v)))

        // Hence conclude our bounds
        ==> !a b. real_interval[a,b] SUBSET {x | x > &0}
                  ==> !i u. u IN real_interval[a,b]
                            ==> tau i u <= max (tau i a) (tau i b) /\
                                taup i u <= max (taup i a) (taup i b)`,
\end{verbatim}\end{scriptsize}

\noindent by combining the original proxy theorem with some basic properties of
posynomials. After some initial breakdown of the goal, also standardizing
inequalities by writing $s > t$ as $t < s$ and so on, we make the trivial
deduction $0 \leq V$ from the assumption $0 < V$ (to settle this in the
hypotheses once and for all for convenient use without explicit mention):

\begin{scriptsize}\begin{verbatim}
  REWRITE_TAC[real_gt; real_ge; GT; GE; LE_0] THEN
  REPEAT GEN_TAC THEN STRIP_TAC THEN
  FIRST_ASSUM(ASSUME_TAC o MATCH_MP REAL_LT_IMP_LE) THEN
  REPEAT GEN_TAC THEN DISCH_TAC THEN
\end{verbatim}\end{scriptsize}

\noindent we first prove that each $\tau_i$ defines a posynomial, by induction:

\begin{scriptsize}\begin{verbatim}
  SUBGOAL_THEN `!i:num. posynomial (tau i)` ASSUME_TAC THENL
   [INDUCT_TAC THEN GEN_REWRITE_TAC RAND_CONV [GSYM ETA_AX] THEN
    ASM_REWRITE_TAC[ADD1; POSYNOMIAL_V] THEN
    MATCH_MP_TAC POSYNOMIAL_ADD THEN CONJ_TAC THENL
     [MATCH_MP_TAC POSYNOMIAL_CMUL THEN ASM_REWRITE_TAC[] THEN
      MATCH_MP_TAC POSYNOMIAL_MUL THEN ASM_SIMP_TAC[ETA_AX];
      MATCH_MP_TAC POSYNOMIAL_CONST THEN
      ASM_MESON_TAC[REAL_LE_TRANS; REAL_ABS_POS; ARITH_RULE `1 <= i + 1`]];
    ALL_TAC] THEN
\end{verbatim}\end{scriptsize}

\noindent and then, using that as a lemma, that the same is true of $\tau^p_i$:

\begin{scriptsize}\begin{verbatim}
  SUBGOAL_THEN `!i:num. posynomial (taup i)` ASSUME_TAC THENL
   [INDUCT_TAC THEN GEN_REWRITE_TAC RAND_CONV [GSYM ETA_AX] THEN
    REWRITE_TAC[ADD1] THEN ONCE_ASM_REWRITE_TAC[] THEN
    MATCH_MP_TAC POSYNOMIAL_MUL THEN REWRITE_TAC[ETA_AX] THEN
    (CONJ_TAC THENL [ALL_TAC; FIRST_X_ASSUM MATCH_ACCEPT_TAC]) THEN
    MATCH_MP_TAC POSYNOMIAL_ADD THEN REWRITE_TAC[POSYNOMIAL_1] THEN
    FIRST_X_ASSUM MATCH_MP_TAC THEN REWRITE_TAC[ETA_AX] THEN
    FIRST_X_ASSUM MATCH_ACCEPT_TAC;
    ALL_TAC] THEN
\end{verbatim}\end{scriptsize}

The result then follows by appealing to a general bound property that the upper
bound of a convex function on a real interval is attained at one of the
endpoints ({\verb|REAL_CONVEX_LOWER_REAL_INTERVAL|}) and the fact that
posynomials are convex functions {\verb|REAL_CONVEX_ON_POSYNOMIAL|} proved at
the end of the previous section:

\begin{scriptsize}\begin{verbatim}
  REPEAT STRIP_TAC THEN
  MATCH_MP_TAC REAL_CONVEX_LOWER_REAL_INTERVAL THEN
  ASM_REWRITE_TAC[] THEN
  FIRST_X_ASSUM(MATCH_MP_TAC o MATCH_MP (REWRITE_RULE[IMP_CONJ_ALT]
        REAL_CONVEX_ON_SUBSET)) THEN
  REWRITE_TAC[GSYM real_gt] THEN MATCH_MP_TAC REAL_CONVEX_ON_POSYNOMIAL THEN
  FIRST_X_ASSUM MATCH_ACCEPT_TAC);;
\end{verbatim}\end{scriptsize}

Before proceeding, for convenience, we collect together a `kitchen sink'
version of the main proxy theorem and corollary together:

\begin{scriptsize}\begin{verbatim}
let FULL_COROLLARY = prove
 (`!(V:real) (beta:num->real) (omega:num->real) (DSF:num->real->real)
    (B:num->real) (H:num->real) (v:num->real) (Tl:num->real) (Tp:num->real)
    (PSI:num->real#real->real) (psi:num->real#real->real)
    (tau:num->real->real) (taup:num->real->real).

        // Environmental assumptions including nondecreasing property
        &0 < V /\
        (!i. i >= 0 ==> beta i > &0) /\
        (!i. i >= 1 ==> (!x. abs (x - DSF i x) <= omega i)) /\
        (!i. i >= 0 ==> abs (psi i (V,Tl i)) <= PSI i (V,abs(Tl i))) /\
        (!i x y. &0 <= x /\ x <= y ==> PSI i (V,x) <= PSI i (V,y)) /\

        (!u. tau 0 u = u) /\
        (!i u. tau (i + 1) u =
               beta (i + 1) * PSI i (u,tau i u) * tau i u + omega (i + 1)) /\
        (!i u. taup i u = (&1 + PSI i (u,tau i u)) * tau i u) /\

        // Computing recursively
        B 0 = &1 /\ H 0 = &0 /\ Tl 0 = V /\

        (!i. Tp i = (&1 + psi i (V,Tl i)) * Tl i) /\
        (!i. v (i + 1) = DSF (i + 1) (beta (i + 1) * Tp i)) /\
        (!i. B (i + 1) = beta (i + 1) * B i) /\
        (!i. H (i + 1) = H i + v (i + 1) / B (i + 1)) /\
        (!i. Tl (i + 1) = beta (i + 1) * Tl i - v (i + 1)) /\

        // The extra posynomial-related assumption
        (!i p. i >= 0 /\ posynomial p
               ==> posynomial (\v. PSI i (v,p v)))

        // Hence conclude invariant and all bounds.
        ==> (!i. V = H i + Tl i / B i) /\
            (!i. abs(Tl i) <= tau i V) /\
            (!i. abs(Tp i) <= taup i V) /\
            (!a b. real_interval[a,b] SUBSET {x | x > &0}
                   ==> !i u. u IN real_interval[a,b]
                             ==> tau i u <= max (tau i a) (tau i b) /\
                                 taup i u <= max (taup i a) (taup i b))`,
\end{verbatim}\end{scriptsize}

The proof is just a trivial if mildly tedious instantiation of earlier results;
this could have been done in one piece at the outset, but we preserved the
separate results from the earlier development:

\begin{scriptsize}\begin{verbatim}
  REWRITE_TAC[real_gt; real_ge; GT; GE; LE_0] THEN
  REPEAT GEN_TAC THEN STRIP_TAC THEN
  FIRST_ASSUM(ASSUME_TAC o MATCH_MP REAL_LT_IMP_LE) THEN
  ONCE_REWRITE_TAC[TAUT `p /\ q /\ r /\ s <=> (p /\ q /\ r) /\ s`] THEN
  CONJ_TAC THENL
   [MATCH_MP_TAC(REWRITE_RULE[GE; LE_0] THEOREM_V_1) THEN
    MAP_EVERY EXISTS_TAC
     [`beta:num->real`; `omega:num->real`; `DSF:num->real->real`;
      `v:num->real`; `PSI:num->real#real->real`;
      `psi:num->real#real->real`] THEN
    ASM_REWRITE_TAC[real_gt];

    MATCH_MP_TAC(REWRITE_RULE[real_gt] COROLLARY_V_3) THEN
    MAP_EVERY EXISTS_TAC
     [`V:real`; `beta:num->real`; `omega:num->real`; `DSF:num->real->real`;
      `B:num->real`; `H:num->real`; `v:num->real`; `Tl:num->real`;
      `Tp:num->real`;
      `PSI:num->real#real->real`; `psi:num->real#real->real`] THEN
    ASM_REWRITE_TAC[GE; LE_0]]);;
\end{verbatim}\end{scriptsize}

\subsection{Instantiation to division (Section~\ref{sec:dsm_division})}

We next proceed with the instantiation to the special cases of division:

\begin{scriptsize}\begin{verbatim}
let BOUND_THEOREM_DIV = prove
 (`!beta Sigma omega B DSF H R Tp X Y g sigma v.
        (!i. i >= 0 ==> beta i > &0) /\
        &1 / &2 <= X /\ X < &1 /\
        &1 <= Y /\ Y < &2 /\
        (!y. &1 <= y /\ y < &2
             ==> g y = (&1 + sigma y) / y /\ abs(sigma y) <= Sigma) /\
        (!i. i >= 1 ==> (!x. abs (x - DSF i x) <= omega i)) /\
        B 0 = &1 /\ H 0 = &0 /\ R 0 = X /\
        (!i. Tp i = g(Y) * R i) /\
        (!i. v (i + 1) = DSF (i + 1) (beta (i + 1) * Tp i)) /\
        (!i. B (i + 1) = beta (i + 1) * B i) /\
        (!i. H (i + 1) = H i + v (i + 1) / B (i + 1)) /\
        (!i. R (i + 1) = beta (i + 1) * R i -  v(i + 1) * Y)
        ==> ?tau. (!u. tau 0 u = u) /\
                  (!i u. tau (i + 1) u =
                         beta (i + 1) * Sigma * tau i u + omega (i + 1)) /\
                  (!i. abs(X / Y - H i)
                       <= max (tau i (&1 / &4)) (tau i (&1)) / B i)`,
\end{verbatim}\end{scriptsize}

We begin by establishing a few obvious facts that we want to avoid re-proving
later such as $0 < B_i$, and deducing that there are indeed functions $\tau$
and $T$ satisfying the recursion equations in the proxy theorem:

\begin{scriptsize}\begin{verbatim}
  REPEAT GEN_TAC THEN REWRITE_TAC[GE; LE_0; real_gt] THEN STRIP_TAC THEN
  SUBGOAL_THEN `&0 <= Sigma` ASSUME_TAC THENL
   [FIRST_X_ASSUM(MP_TAC o SPEC `&1:real`) THEN REAL_ARITH_TAC;
    ALL_TAC] THEN
  SUBGOAL_THEN `!i. &0 < (B:num->real) i` ASSUME_TAC THENL
   [INDUCT_TAC THEN ASM_SIMP_TAC[REAL_LT_MUL; ADD1; REAL_LT_01]; ALL_TAC] THEN
  SUBGOAL_THEN `&0 < X /\ &0 < Y` STRIP_ASSUME_TAC THENL
   [ASM_REAL_ARITH_TAC; ALL_TAC] THEN
  SUBGOAL_THEN `&0 < X / Y` ASSUME_TAC THENL
   [ASM_MESON_TAC[REAL_LT_DIV]; ALL_TAC] THEN
  MAP_EVERY ABBREV_TAC
   [`PSI:num->real#real->real = \i (u,t). Sigma`;
    `psi:num->real#real->real = \i (u,t). sigma(Y:real)`] THEN
  (X_CHOOSE_THEN `tau:num->real->real`
    (STRIP_ASSUME_TAC o REWRITE_RULE[ADD1]) o
   prove_recursive_functions_exist num_RECURSION)
    `(!u:real. tau 0 u = u) /\
     (!i u. tau (SUC i) u =
            beta (i + 1) * PSI i (u,tau i u) * tau i u + omega (i + 1))` THEN
  (X_CHOOSE_THEN `Tl:num->real`
    (STRIP_ASSUME_TAC o REWRITE_RULE[ADD1]) o
   prove_recursive_functions_exist num_RECURSION)
    `Tl 0 :real = X / Y /\
    !i. Tl (SUC i) = beta (i + 1) * Tl i - v (i + 1)` THEN
  ABBREV_TAC
   `taup:num->real->real = \i u. (&1 + PSI i (u,tau i u)) * tau i u` THEN
\end{verbatim}\end{scriptsize}

We then simply instantiate the proxy theorem/corollary appropriately:

\begin{scriptsize}\begin{verbatim}
  MP_TAC(ISPECL
   [`X / Y:real`;
    `beta:num->real`;
    `omega:num->real`;
    `DSF:num->real->real`;
    `B:num->real`;
    `H:num->real`;
    `v:num->real`;
    `Tl:num->real`;
    `Tp:num->real`;
    `PSI:num->real#real->real`;
    `psi:num->real#real->real`;
    `tau:num->real->real`;
    `taup:num->real->real`]
    FULL_COROLLARY) THEN
\end{verbatim}\end{scriptsize}

Now after some trivial cleanup and splitting

\begin{scriptsize}\begin{verbatim}
  REWRITE_TAC[GE; LE_0; real_gt] THEN ANTS_TAC THENL
\end{verbatim}\end{scriptsize}

\noindent we first need to verify the various hypotheses of the proxy theorem
and corollary. In all we get 17(!) of them. However, it turns out that
most have trivial one-line
proofs like {\verb|FIRST_X_ASSUM MATCH_ACCEPT_TAC|}. The only one with a little
content is proving that {\verb|!i. Tp i = (&1 + psi i (X / Y,Tl i)) * Tl i|}.
After a little initial rearrangement this devolves to proving
{\verb|!j. R j / Y = Tl j|}, which is done by an easy induction (this
corresponds to verifying the equivalence of $R$ and $\tilde{R}$ in the text).
Now we have the conclusions from the main theorem/corollary and we do some
instantiation, in particular setting the endpoints of the interval for which
the bound is derived, and hence derive our result:

\begin{scriptsize}\begin{verbatim}
    STRIP_TAC THEN EXISTS_TAC `tau:num->real->real` THEN
    ASM_REWRITE_TAC[REAL_ADD_SUB] THEN CONJ_TAC THENL
     [EXPAND_TAC "PSI" THEN REWRITE_TAC[]; ALL_TAC] THEN
    ASM_SIMP_TAC[REAL_ABS_DIV; REAL_LE_DIV2_EQ;
                 REAL_ARITH `&0 < b ==> abs b = b`] THEN
    X_GEN_TAC `i:num` THEN
    FIRST_X_ASSUM(MP_TAC o SPECL [`&1 / &4`; `&1`]) THEN
    REWRITE_TAC[SUBSET; IN_REAL_INTERVAL; IN_ELIM_THM] THEN
    ANTS_TAC THENL [REAL_ARITH_TAC; ALL_TAC] THEN
    DISCH_THEN(MP_TAC o SPECL [`i:num`; `X / Y:real`]) THEN
    ANTS_TAC THENL [ALL_TAC; ASM_MESON_TAC[REAL_LE_TRANS]] THEN
    REWRITE_TAC[REAL_ARITH
     `&1 / &4 <= X / Y /\ X / Y <= &1 <=>
      &1 / &2 * inv(&2) <= X * inv Y /\ X * inv Y <= &1 * inv(&1)`] THEN
    CONJ_TAC THEN MATCH_MP_TAC REAL_LE_MUL2 THEN REPEAT CONJ_TAC THEN
    TRY(MATCH_MP_TAC REAL_LE_INV2) THEN
    REWRITE_TAC[REAL_LE_INV_EQ] THEN ASM_REAL_ARITH_TAC]);;
\end{verbatim}\end{scriptsize}

\subsection{Instantiation to square root (Section~\ref{sec:dsm_squareroot})}

This is conceptually the same as the instantiation to division, but various
terms become more involved and as a result the proof becomes a bit more
complicated too.

\begin{scriptsize}\begin{verbatim}
let BOUND_THEOREM_SQRT = prove
 (`!beta Sigma omega B DSF H R Tp X g sigma v.
        (!i. i >= 0 ==> beta i > &0) /\
        &1 / &4 <= X /\ X < &1 /\
        (!x. &1 / &4 <= x /\ x < &1
             ==> g x = (&1 + sigma x) / sqrt x /\
                 abs(sigma x) <= Sigma) /\
        (!i. i >= 1 ==> (!x. abs (x - DSF i x) <= omega i)) /\
        B 0 = &1 /\ H 0 = &0 /\ R 0 = X / &2 /\
        (!i. Tp i = (if i = 0 then &2 else &1) * g(X) * R i) /\
        (!i. v (i + 1) = DSF (i + 1) (beta (i + 1) * Tp i)) /\
        (!i. B (i + 1) = beta (i + 1) * B i) /\
        (!i. H (i + 1) = H i + v (i + 1) / B (i + 1)) /\
        (!i. R (i + 1) =
             beta (i + 1) * R i -  v(i + 1) * (H(i + 1) + H i) / &2)
        ==> ?tau.
                  (!u. tau 0 u = u) /\
                  (!i u. tau (i + 1) u =
                         beta (i + 1) *
                         (if i = 0 then Sigma
                          else Sigma + (&1 + Sigma) * tau i u / (&2 * u * B i))
                          * tau i u +
                         omega (i + 1)) /\
                  (!i. abs(sqrt X - H i)
                       <= max (tau i (&1 / &2)) (tau i (&1)) / B i)`,
\end{verbatim}\end{scriptsize}

As before we start by establishing some basic lemmas and the existence of
recursively defined functions:

\begin{scriptsize}\begin{verbatim}
  REPEAT GEN_TAC THEN REWRITE_TAC[GE; LE_0; real_gt] THEN STRIP_TAC THEN
  SUBGOAL_THEN `&0 <= Sigma` ASSUME_TAC THENL
   [FIRST_X_ASSUM(MP_TAC o SPEC `&1 / &2`) THEN REAL_ARITH_TAC;
    ALL_TAC] THEN
  SUBGOAL_THEN `!i. &0 < (B:num->real) i` ASSUME_TAC THENL
   [INDUCT_TAC THEN ASM_SIMP_TAC[REAL_LT_MUL; ADD1; REAL_LT_01]; ALL_TAC] THEN
  SUBGOAL_THEN `&0 < X` ASSUME_TAC THENL
   [ASM_REAL_ARITH_TAC; ALL_TAC] THEN
  SUBGOAL_THEN `&0 < sqrt X` ASSUME_TAC THENL
   [ASM_MESON_TAC[SQRT_POS_LT]; ALL_TAC] THEN
  MAP_EVERY ABBREV_TAC
   [`PSI:num->real#real->real = \i (u,t).
        if i = 0 then Sigma
        else Sigma + (&1 + Sigma) * t / (&2 * u * B i)`;
    `psi:num->real#real->real = \i (u,t).
        if i = 0 then sigma(X)
        else (&1 + sigma(X:real)) * (&1 - t / (&2 * u * B i)) - &1`] THEN
  (X_CHOOSE_THEN `tau:num->real->real`
    (STRIP_ASSUME_TAC o REWRITE_RULE[ADD1]) o
   prove_recursive_functions_exist num_RECURSION)
    `(!u:real. tau 0 u = u) /\
     (!i u. tau (SUC i) u =
            beta (i + 1) * PSI i (u,tau i u) * tau i u + omega (i + 1))` THEN
  (X_CHOOSE_THEN `Tl:num->real`
    (STRIP_ASSUME_TAC o REWRITE_RULE[ADD1]) o
   prove_recursive_functions_exist num_RECURSION)
    `Tl 0 = sqrt(X) /\
    !i. Tl (SUC i) = beta (i + 1) * Tl i - v (i + 1)` THEN
  ABBREV_TAC
   `taup:num->real->real = \i u. (&1 + PSI i (u,tau i u)) * tau i u` THEN
\end{verbatim}\end{scriptsize}

\noindent and then instantiate the proxy theorem/corollary:

\begin{scriptsize}\begin{verbatim}
  MP_TAC(ISPECL
   [`sqrt X`;
    `beta:num->real`;
    `omega:num->real`;
    `DSF:num->real->real`;
    `B:num->real`;
    `H:num->real`;
    `v:num->real`;
    `Tl:num->real`;
    `Tp:num->real`;
    `PSI:num->real#real->real`;
    `psi:num->real#real->real`;
    `tau:num->real->real`;
    `taup:num->real->real`]
    FULL_COROLLARY) THEN
  REWRITE_TAC[GE; LE_0; real_gt] THEN ANTS_TAC THENL
\end{verbatim}\end{scriptsize}

The establishment of the hypotheses is now more complicated, mainly because of
the more intricate proof that $R = \tilde{R}$.

\begin{scriptsize}\begin{verbatim}
   [REPEAT CONJ_TAC THENL
     [FIRST_X_ASSUM MATCH_ACCEPT_TAC;
      FIRST_X_ASSUM MATCH_ACCEPT_TAC;
      FIRST_X_ASSUM MATCH_ACCEPT_TAC;
      X_GEN_TAC `i:num` THEN MAP_EVERY EXPAND_TAC ["PSI"; "psi"] THEN
      REWRITE_TAC[] THEN ASM_CASES_TAC `i = 0` THEN ASM_SIMP_TAC[] THEN
      MATCH_MP_TAC(REAL_ARITH
       `abs x <= a /\ abs((&1 + x) * y) <= b
        ==> abs((&1 + x) * (&1 - y) - &1) <= a + b`) THEN
      ASM_SIMP_TAC[REAL_ABS_MUL] THEN
      MATCH_MP_TAC REAL_LE_MUL2 THEN REWRITE_TAC[REAL_ABS_POS] THEN
      ASM_SIMP_TAC[REAL_ARITH `abs x <= a ==> abs(&1 + x) <= &1 + a`] THEN
      REWRITE_TAC[REAL_ABS_DIV] THEN MATCH_MP_TAC REAL_EQ_IMP_LE THEN
      AP_TERM_TAC THEN
      MATCH_MP_TAC(REAL_ARITH `&0 < x ==> abs(&2 * x) = &2 * x`) THEN
      MATCH_MP_TAC REAL_LT_MUL THEN ASM_REWRITE_TAC[];
      MAP_EVERY X_GEN_TAC [`i:num`; `x:real`; `y:real`] THEN STRIP_TAC THEN
      EXPAND_TAC "PSI" THEN REWRITE_TAC[] THEN
      COND_CASES_TAC THEN ASM_REWRITE_TAC[REAL_LE_REFL; REAL_LE_LADD] THEN
      ASM_SIMP_TAC[REAL_LE_LADD; REAL_LE_LMUL_EQ; REAL_LE_DIV2_EQ;
                   REAL_ARITH `&0 <= s ==> &0 < &1 + s`; REAL_LT_MUL;
                   REAL_ARITH `&0 < &2 * x <=> &0 < x`] THEN
      REAL_ARITH_TAC;
      FIRST_X_ASSUM MATCH_ACCEPT_TAC;
      ASM_REWRITE_TAC[] THEN NO_TAC;
      EXPAND_TAC "taup" THEN REWRITE_TAC[] THEN NO_TAC;
      FIRST_X_ASSUM MATCH_ACCEPT_TAC;
      FIRST_X_ASSUM MATCH_ACCEPT_TAC;
      FIRST_X_ASSUM MATCH_ACCEPT_TAC;
      X_GEN_TAC `i:num` THEN
      FIRST_X_ASSUM(fun th -> GEN_REWRITE_TAC LAND_CONV [th]) THEN
      EXPAND_TAC "psi" THEN REWRITE_TAC[] THEN
      ASM_CASES_TAC `i = 0` THEN ASM_REWRITE_TAC[] THENL
       [ASM_SIMP_TAC[REAL_DIV_SQRT; REAL_LT_IMP_LE; REAL_ARITH
         `&2 * c / s * x / &2 = c * x / s`];
        ALL_TAC] THEN
      REWRITE_TAC[REAL_MUL_LID; REAL_ARITH `&1 + x - &1 = x`] THEN
      ASM_SIMP_TAC[] THEN REWRITE_TAC[real_div; GSYM REAL_MUL_ASSOC] THEN
      AP_TERM_TAC THEN MATCH_MP_TAC(REAL_FIELD
       `&0 < b /\ &0 < s /\ r = (s - t / b / &2) * t
        ==> inv s * r = (&1 - t * inv(&2 * s * b)) * t`) THEN
      ASM_REWRITE_TAC[] THEN
      SUBGOAL_THEN `!j:num. Tl j / B j = sqrt X - H j`
      ASSUME_TAC THENL
       [INDUCT_TAC THEN ASM_REWRITE_TAC[REAL_SUB_RZERO; REAL_DIV_1; ADD1] THEN
        UNDISCH_TAC `Tl(j:num) / B j = sqrt X - H j` THEN
        SUBGOAL_THEN `&0 < beta(j + 1) /\ &0 < B j` MP_TAC THENL
         [ASM_REWRITE_TAC[]; CONV_TAC REAL_FIELD];
        ASM_REWRITE_TAC[REAL_ARITH `s - (s - h) / &2 = (s + h) / &2`]] THEN
      MATCH_MP_TAC(REAL_FIELD
       `!b. &0 < b /\ x / b = y / &2 * z / b ==> x = y / &2 * z`) THEN
      EXISTS_TAC `(B:num->real) i` THEN ASM_REWRITE_TAC[REAL_ARITH
       `(x + h) / &2 * (x - h) = (x pow 2 - h pow 2) / &2`] THEN
      ASM_SIMP_TAC[SQRT_POW_2; REAL_LT_IMP_LE] THEN
      ASM_SIMP_TAC[REAL_EQ_LDIV_EQ] THEN
      SPEC_TAC(`i:num`,`j:num`) THEN
      MATCH_MP_TAC num_INDUCTION THEN CONJ_TAC THENL
       [ASM_REWRITE_TAC[] THEN REAL_ARITH_TAC; REWRITE_TAC[ADD1]] THEN
      ONCE_REWRITE_TAC[ASSUME
      `!i. R (i + 1) =
           beta (i + 1) * R i - v (i + 1) * (H (i + 1) + H i) / &2`] THEN
      X_GEN_TAC `j:num` THEN SIMP_TAC[] THEN
      REWRITE_TAC[ASSUME
       `!i. H (i + 1):real = H i + v (i + 1) / B (i + 1)`] THEN
      REWRITE_TAC[ASSUME `!i. B (i + 1):real = beta (i + 1) * B i`] THEN
      SUBGOAL_THEN `&0 < beta(j + 1) /\ &0 < B j` MP_TAC THENL
       [ASM_REWRITE_TAC[]; CONV_TAC REAL_FIELD];
      FIRST_X_ASSUM MATCH_ACCEPT_TAC;
      FIRST_X_ASSUM MATCH_ACCEPT_TAC;
      FIRST_X_ASSUM MATCH_ACCEPT_TAC;
      FIRST_X_ASSUM MATCH_ACCEPT_TAC;
      MAP_EVERY X_GEN_TAC [`i:num`; `p:real->real`] THEN DISCH_TAC THEN
      EXPAND_TAC "PSI" THEN REWRITE_TAC[] THEN
      ASM_CASES_TAC `i = 0` THEN ASM_SIMP_TAC[POSYNOMIAL_CONST] THEN
      MATCH_MP_TAC POSYNOMIAL_ADD THEN
      ASM_SIMP_TAC[POSYNOMIAL_CONST] THEN
      MATCH_MP_TAC POSYNOMIAL_MUL THEN
      ASM_SIMP_TAC[POSYNOMIAL_CONST; REAL_ARITH
       `&0 <= s ==> &0 <= &1 + s`] THEN
      REWRITE_TAC[real_div; REAL_INV_MUL] THEN REWRITE_TAC[ REAL_ARITH
       `x * inv(&2) * inv y * z = (inv(&2) * z) * x / y`] THEN
      MATCH_MP_TAC POSYNOMIAL_CMUL THEN
      ASM_SIMP_TAC[REAL_LT_INV_EQ; REAL_ARITH
       `&0 < inv(&2) * x <=> &0 < x`] THEN
      MATCH_MP_TAC POSYNOMIAL_VDIV THEN ASM_REWRITE_TAC[]];
\end{verbatim}\end{scriptsize}

The use of the result is very similar, however, and this quickly concludes
the proof:

\begin{scriptsize}\begin{verbatim}
    STRIP_TAC THEN EXISTS_TAC `tau:num->real->real` THEN
    ASM_REWRITE_TAC[REAL_ADD_SUB] THEN CONJ_TAC THENL
     [EXPAND_TAC "PSI" THEN REWRITE_TAC[]; ALL_TAC] THEN
    ASM_SIMP_TAC[REAL_ABS_DIV; REAL_LE_DIV2_EQ;
                 REAL_ARITH `&0 < b ==> abs b = b`] THEN
    X_GEN_TAC `i:num` THEN
    FIRST_X_ASSUM(MP_TAC o SPECL [`&1 / &2`; `&1`]) THEN
    REWRITE_TAC[SUBSET; IN_REAL_INTERVAL; IN_ELIM_THM] THEN
    ANTS_TAC THENL [REAL_ARITH_TAC; ALL_TAC] THEN
    DISCH_THEN(MP_TAC o SPECL [`i:num`; `sqrt X`]) THEN
    ANTS_TAC THENL [ALL_TAC; ASM_MESON_TAC[REAL_LE_TRANS]] THEN
    CONJ_TAC THENL
     [MATCH_MP_TAC REAL_LE_RSQRT; MATCH_MP_TAC REAL_LE_LSQRT] THEN
    ASM_REAL_ARITH_TAC]);;
\end{verbatim}\end{scriptsize}

\subsection{Automated instantiation (related to Table~\ref{tab:dsm_div_sqrt})}

For convenience, we have implemented a HOL Light derived rule to instantiate
the parameters of the theorems for division and square root and derive
appropriately accurate error bounds for the successive approximations. A HOL
Light derived rule is essentially a programmatic combination of more basic
rules of inference, which is still doing full logical proof behind the scenes.
Thus we can consider this as analogous to a spreadsheet producing results
automatically as parameters are varied, but with the additional security of
{\em proving} the result. We will not discuss the coding in detail, but it is
very standard for such applications and can be understood by manually tracing
through specific examples.

\begin{scriptsize}\begin{verbatim}
let BOUNDS_INSTATIATION =
  let pth = prove
     (`x <= a / b ==> &0 <= b ==> !a'. a <= a' ==> x <= a' / b`,
      REPEAT STRIP_TAC THEN TRANS_TAC REAL_LE_TRANS `a / b:real` THEN
      ASM_REWRITE_TAC[] THEN REWRITE_TAC[real_div] THEN
      MATCH_MP_TAC REAL_LE_RMUL THEN ASM_REWRITE_TAC[REAL_LE_INV_EQ]) in
  let rec calc rews (thb,ths) n =
    if n = 0 then [thb] else
    let oths = calc rews (thb,ths) (n - 1) in
    let th1 = CONV_RULE NUM_REDUCE_CONV (SPEC(mk_small_numeral(n - 1)) ths) in
    let th2 = GEN_REWRITE_RULE TOP_DEPTH_CONV (hd oths::rews) th1 in
    let th3 = CONV_RULE REAL_RAT_REDUCE_CONV th2 in
    th3::oths in
  fun th beta sigma omega n d ->
    let ith = BETA_RULE (SPECL [beta; sigma; omega] th) in
    let avs,itm = strip_forall(concl ith) in
    let hth = ASSUME (rand(lhand itm)) in
    let eth = MP (SPECL avs ith) (CONJ (REAL_ARITH(lhand(lhand itm))) hth) in
    let ev,ebod = dest_exists(concl eth) in
    let [th0;th1;bth] = CONJUNCTS(ASSUME ebod) in
    let (th_b,th_s) =
      let hths = CONJUNCTS hth in
      el (if th = BOUND_THEOREM_DIV then 6 else 4) hths,
      el (if th = BOUND_THEOREM_DIV then 11 else 9) hths in
    let bths = calc [] (th_b,th_s) n in
    let tths_lo =
      calc bths (SPEC (if th = BOUND_THEOREM_DIV then `&1 / &4` else `&1 / &2`)
                 th0,
            SPEC (if th = BOUND_THEOREM_DIV then `&1 / &4` else `&1 / &2`)
                 (GEN_REWRITE_RULE I [SWAP_FORALL_THM] th1)) n
    and tths_hi =
      calc bths (SPEC `&1:real` th0,
            SPEC `&1:real` (GEN_REWRITE_RULE I [SWAP_FORALL_THM] th1)) n in
    let aths = map
     (CONV_RULE REAL_RAT_REDUCE_CONV o
      REWRITE_RULE(tths_lo@tths_hi) o
      C SPEC bth o mk_small_numeral) (0--n) in
    let weaken th =
      let th1 = MATCH_MP pth th in
      let th2 = GEN_REWRITE_CONV RAND_CONV bths (lhand(concl th1)) in
      let th3 = CONV_RULE(RAND_CONV REAL_RAT_REDUCE_CONV) th2 in
      let th4 = MP th1 (EQT_ELIM th3) in
      let rr = rat_of_term(lhand(lhand(snd(dest_forall(concl th4))))) in
      let yy = pow10 d in
      let xx = ceiling_num(yy */ rr) in
      let th5 = SPECL [mk_numeral xx; mk_numeral yy] DECIMAL in
      let th6 = SPEC (lhand(concl th5)) th4 in
      MP th6 (EQT_ELIM(REAL_RAT_REDUCE_CONV(lhand(concl th6)))) in
    let ath = end_itlist CONJ (map weaken aths) in
    GENL avs (DISCH_ALL (CHOOSE(ev,eth) ath));;
\end{verbatim}\end{scriptsize}

The toplevel function takes a number of parameters

\begin{itemize}

\item {\tt th} is the bounds theorem to instantiate, which will be
{\verb|BOUND_THEOREM_DIV|} or {\verb|BOUND_THEOREM_SQRT|}.

\item {\tt beta}, {\tt sigma} and {\tt omega} are HOL term instantiations for
the particular values of $\beta$, $\Sigma$ and $\Omega$.

\item {\tt n} is the number of iterations for which bounds are desired: an
input of $n$ will result in bounds for $H_0,H_1,\ldots,H_n$.

\item {\tt d} is the number of fractional digits in the decimal representation
of the digit bounds.

\end{itemize}

For example the instantiation:

\begin{scriptsize}\begin{verbatim}
BOUNDS_INSTATIATION BOUND_THEOREM_SQRT
 `(\i. if i = 2 then &32 else if i = 5 then &64 else &128):num->real`
 `inv(&2 pow 8):real`
 `(\i. if i = 0 then &1 / &2 else &9 / &16):num->real`
 7 6;;
\end{verbatim}\end{scriptsize}

\noindent results automatically in the following theorem giving bounds to 6
places after the decimal point for the iterations $H_0,\ldots,H_7$ for the
square root algorithm with (somewhat arbitrary) parameters:

\begin{scriptsize}\begin{verbatim}
  |- !B DSF H R Tp X g sigma v.
         &1 / &4 <= X /\
         X < &1 /\
         (!x. &1 / &4 <= x /\ x < &1
              ==> g x = (&1 + sigma x) / sqrt x /\
                  abs (sigma x) <= inv (&2 pow 8)) /\
         (!i. i >= 1
              ==> (!x. abs (x - DSF i x) <=
                       (if i = 0 then &1 / &2 else &9 / &16))) /\
         B 0 = &1 /\
         H 0 = &0 /\
         R 0 = X / &2 /\
         (!i. Tp i = (if i = 0 then &2 else &1) * g X * R i) /\
         (!i. v (i + 1) =
              DSF (i + 1)
              ((if i + 1 = 2 then &32 else if i + 1 = 5 then &64 else &128) *
               Tp i)) /\
         (!i. B (i + 1) =
              (if i + 1 = 2 then &32 else if i + 1 = 5 then &64 else &128) *
              B i) /\
         (!i. H (i + 1) = H i + v (i + 1) / B (i + 1)) /\
         (!i. R (i + 1) =
              (if i + 1 = 2 then &32 else if i + 1 = 5 then &64 else &128) *
              R i -
              v (i + 1) * (H (i + 1) + H i) / &2)
         ==> abs (sqrt X - H 0) <= #1.000000 / B 0 /\
             abs (sqrt X - H 1) <= #1.062500 / B 1 /\
             abs (sqrt X - H 2) <= #0.836978 / B 2 /\
             abs (sqrt X - H 3) <= #0.998973 / B 3 /\
             abs (sqrt X - H 4) <= #1.062231 / B 4 /\
             abs (sqrt X - H 5) <= #0.828059 / B 5 /\
             abs (sqrt X - H 6) <= #0.976530 / B 6 /\
             abs (sqrt X - H 7) <= #1.050765 / B 7
\end{verbatim}\end{scriptsize}

Using similar simple invocations we can exactly check the main bounds given in
Table 1. Where they differ in the last digit, the difference arises because our
theorems are returning actual bounds whereas the table just rounds the bounds
to nearest.


\begin{thebibliography}{10}

\bibitem{barvinok-convexity}
Alexander Barvinok.
\newblock {\em A Course in Convexity}, volume~54 of {\em Graduate Texts in
  Mathematics}.
\newblock American Mathematical Society, 2002.

\bibitem{boyd2004convex}
Stephen Boyd and Lieven Vandenberghe.
\newblock {\em Convex Optimization}.
\newblock Cambridge University Press, 2004.

\bibitem{briggs199317}
W~S Briggs and David~W Matula.
\newblock A 17$\times$ 69 bit multiply and add unit with redundant binary
  feedback and single cycle latency.
\newblock In {\em Proceedings of the 11th Symposium on Computer Arithmetic},
  pages 163--170. IEEE, 1993.

\bibitem{daumas1997recoders}
M~Daumas and D~Matula.
\newblock Recoders for partial compression and rounding.
\newblock {\em Ecole Normale Sup\'{e}rieure de Lyon, Research Report
  LIP-RR1997-01}, 1997.

\bibitem{daumas2003further}
Marc Daumas and David~W Matula.
\newblock Further reducing the redundancy of a notation over a minimally
  redundant digit set.
\newblock {\em The Journal of VLSI Signal Processing}, 33(1):7--18, 2003.

\bibitem{dufin1967geometric}
Richard~J Dufin, Elmor~L Peterson, and Clarence Zener.
\newblock {\em Geometric programming-theory and application}.
\newblock John Wiley, 1967.

\bibitem{ercegovac1987fly}
Milos~D Ercegovac and Tomas Lang.
\newblock On-the-fly conversion of redundant into conventional representations.
\newblock {\em IEEE Transactions on Computers}, 36(7):895--897, 1987.

\bibitem{ercegovac1989fly}
Milos~D Ercegovac and Tomas Lang.
\newblock On-the-fly rounding for division and square root.
\newblock In {\em Proceedings of 9th IEEE Symposium on Computer Arithmetic},
  pages 169--173. IEEE, 1989.

\bibitem{ercegovac1994division}
Milos~D Ercegovac and Tomas Lang.
\newblock {\em Division and Square Root: Digit-Recurrence Algorithms and
  Implementations}.
\newblock Kluwer Academic Publishers, 1994.

\bibitem{ercegovac1994recoding}
Milos~D Ercegovac and Tom{\'a}s Lang.
\newblock On recoding in arithmetic algorithms.
\newblock In {\em Conference Record of the Twenty-Eighth Asilomar Conference on
  Signals, Systems and Computers}, volume~1, pages 531--535. IEEE, 1994.

\bibitem{ercegovac1994very}
Milos~D Ercegovac, Tomas Lang, and Paolo Montuschi.
\newblock Very-high radix division with prescaling and selection by rounding.
\newblock {\em IEEE Transactions on Computers}, 43(8):909--918, 1994.

\bibitem{harrison1996hol}
John Harrison.
\newblock {H}{O}{L} {L}ight: {A} tutorial introduction.
\newblock In {\em Proceedings of the First International Conference on Formal
  Methods in Computer-Aided Design (FMCAD96)}, volume 1166 of Lecture Notes in
  Computer Science, pages 265--269. Springer, 1996.

\bibitem{harrison-tutorial}
John Harrison.
\newblock The {HOL} {L}ight tutorial.
\newblock Unpublished manual available at
  {\url{http://www.cl.cam.ac.uk/~jrh13/hol-light/tutorial.pdf}}, 2006.

\bibitem{KnuthArt2}
Donald~E. Knuth.
\newblock {\em {Art of Computer Programming, Volume 2: Seminumerical
  Algorithms}}.
\newblock Addison-Wesley Professional, third edition, November 1997.

\bibitem{lang1995very}
Tomas Lang and Paolo Montuschi.
\newblock Very-high radix combined division and square root with prescaling and
  selection by rounding.
\newblock In {\em Proceedings of the 12th IEEE Symposium on Computer
  Arithmetic}, pages 124--131. IEEE, 1995.

\bibitem{niculescu2006convex}
Constantin Niculescu and Lars-Erik Persson.
\newblock {\em Convex Functions and their Applications: A Contemporary
  Approach}.
\newblock Springer Science \& Business Media, 2006.

\bibitem{webster-convexity}
Roger Webster.
\newblock {\em Convexity}.
\newblock Oxford University Press, 1995.

\end{thebibliography}
\end{document}